\documentclass[a4paper,10pt,reqno]{amsart}

\newtheorem{thm}{Theorem}[section]
\newtheorem{lem}[thm]{Lemma}

\newtheorem{prop}[thm]{Proposition}
\theoremstyle{definition}
\newtheorem{defn}[thm]{Definition}

\newtheorem{rem}[thm]{Remark}
\newtheorem{con}[thm]{Conjecture}

\begin{document}

\title{List colourings of multipartite hypergraphs}

\subjclass[2000]{05C15}

\author{Ar\`es M\'eroueh}\thanks{The first author was supported by the EPSRC}
\author{Andrew Thomason}

\address{Department of Pure Mathematics and Mathematical Statistics\\
Centre for Mathematical Sciences, Wilberforce Road, Cambridge CB3 0WB, UK}

\email{ajm271@dpmms.cam.ac.uk}
\email{a.g.thomason@dpmms.cam.ac.uk}

\begin{abstract}
  Let $\chi_l(G)$ denote the list chromatic number of the $r$-uniform
  hypergraph~$G$. Extending a result of Alon for graphs, Saxton and the
  second author used the method of containers to prove that, if $G$ is
  simple and $d$-regular, then  $\chi_l(G)\ge (1/(r-1)+o(1))\log_r d$.

  To see how close this inequality is to best possible, we examine
  $\chi_l(G)$ when $G$ is a random $r$-partite hypergraph with $n$ vertices
  in each class. The value when $r=2$ was determined by Alon and
  Krivelevich; here we show that $\chi_l(G)= (g(r,\alpha)+o(1))\log_r d$
  almost surely, where $d$ is the expected average degree of~$G$ and
  $\alpha=\log_nd$.

  The function $g(r,\alpha)$ is defined in terms of ``preference orders''
  and can be determined fairly explicitly. This is enough to show that the
  container method gives an optimal lower bound on $\chi_l(G)$ for $r=2$
  and $r=3$, but, perhaps surprisingly, apparently not for $r\ge4$.
\end{abstract}

\maketitle

\section{Introduction}

Let $G$ be an $r$-uniform hypergraph: that is to say, its edges are sets of
$r$ vertices. For brevity, we often call $G$ an $r$-graph: thus a 2-graph
is just a graph. Given an assignment $L:V(G)\to\mathcal{P}(\mathbb{N})$ of
a list $L(v)$ of colours to each vertex~$v$, we say $G$ is $L$-{\em
  chooseable} if, for each vertex~$v$, it is possible to choose a colour
$c(v)\in L(v)$, such that there is no edge~$e$ with $c(v)$ the same for all
$v\in e$. The minimum number~$k$ such that $G$ is $L$-choosable whenever
$|L(v)|\ge k$ for every~$v$, is called the {\em list-chromatic} number
of~$G$, denoted by $\chi_l(G)$. This notion was introduced for graphs by
Vizing~\cite{V} and by Erd\H{o}s, Rubin and
Taylor~\cite{ERT}. In~\cite{ERT} it was proved, amongst other things, that
$\chi_l(K_{d,d})=(1+o(1))\log_2 d$, and also that the determination of
$\chi_l(K_{d,d})$ is intimately related to the study of ``Property~B''
(namely, the study of the minimum number of edges in a non-bipartite
uniform hypergraph). The $o(1)$ term here, as elsewhere in this paper,
denotes a quantity tending to zero as $d\to\infty$.  (This is a convenient
place to point out that logarithms to various different bases appear in
this paper but, where no base is specified, the logarithm is natural.)

The theorem of Erd\H{o}s, Rubin and Taylor was extended by Alon and
Krivelevich~\cite{AKr}, who proved that $\chi_l(G)=(1+o(1))\log_2 d$ holds
almost surely for a random bipartite graph with~$n$ vertices in each class,
edges being present independently with probability~$p$, provided $d=pn>d_0$
for some constant~$d_0$. (They actually proved something a little sharper,
and they showed that this more general result is also tied to Property~B.)

Alon~\cite{A2} proved that {\em every} graph $G$ of average degree~$d$
satisfies $\chi_l(G)\ge (1/2+o(1))\log_2 d$. The value $1/2$ can, in fact,
be replaced by~$1$ here (see below), and so it follows that complete
bipartite graphs, and more generally random bipartite graphs, are graphs
whose list chromatic number is (more or less) minimal amongst graphs of
given average degree.

When $r\ge3$ it is not true, in general, that the list chromatic number of
an $r$-graph grows with its average degree. For example, if $F$ is a
2-graph and $G$ is an $r$-graph on the same vertex set, such that every
edge of $G$ contains an edge of~$F$, then $\chi_l(G)\le \chi_l(F)$, but the
average degree of $G$ can be large, even if that of $F$ is not and
$\chi_l(F)$ is small. However, examples of this kind can be avoided by
considering {\em simple} $r$-graphs, in which different edges have at most
one vertex in common. For this reason, we are particularly interested in
simple hypergraphs.

The case when the edges of $G$ form a Steiner triple system was studied by
Haxell and Pei~\cite{HP}, who proved $\chi_l(G)=\Omega(\log d/\log \log
d)$. Haxell and Verstra\"ete~\cite{HV} obtained the bound
$\chi_l(G)\ge(1+o(1))\left(\log d/5\log\log d\right)^{1/2}$ for every
simple $d$-regular $3$-graph, and Alon and Kostochka~\cite{AK1} showed that
$\chi_l(G)\ge(\log d)^{1/(r-1)}$ for every simple $r$-graph~$G$ of average
degree~$d$; in particular $\chi_l(G)$ grows with~$d$.

The correct rate of growth was found by Saxton and Thomason~\cite{ST1}, who
proved, via the container method, that $\chi_l(G)\ge\Omega(\log d)$ if $G$
is a simple $d$-regular $r$-graph (see~\cite{ST3} for a refinement of the
argument). An improved bound was obtained in~\cite{ST2}, with an extension
to not-quite-simple $r$-graphs in~\cite{ST4}: we state it here.

\begin{prop}[\cite{ST2,ST4}]\label{cor:chil}
  Let $r\in\mathbb N$ be fixed. Let $G$ be an $r$-graph with average
  degree~$d$. Suppose that, for $2\le j\le r$, each set of $j$ vertices
  lies in at most $d^{(r-j)/(r-1)+o(1)}$ edges, where $o(1)\to0$ as
  $d\to\infty$. Then
  $$
  \chi_l(G)\,\ge\,\,(1+o(1))\,\frac{1}{(r-1)^2} \log_rd \,.
  $$
  Moreover, if $G$ is regular then
  $$
  \chi_l(G)\,\ge\,\,(1+o(1))\,\frac{1}{r-1}\log_rd\,.
  $$
\end{prop}

In particular, for $2$-graphs, the $1/2$ in Alon's bound can be replaced
by~$1$, which is tight, as described above. Thus the container method gives
a best possible bound for graphs. Does it also give an optimal bound for
$r$-graphs (at least simple $r$-graphs) when $r\ge 3$? That is the question
underlying the results of this paper.

To answer this question, it is natural, in the light of what is known about
$2$-graphs, to examine $r$-partite $r$-graphs. (Every $r$-graph contains an
$r$-partite subgraph whose average degree is less by only a constant
factor, so a lower bound on $\chi_l$ for $r$-partite graphs applies to all
$r$-graphs. On the other hand, non-$r$-partite $d$-regular simple
$r$-graphs can have $\chi_l$ as large as $\Omega(d/\log d)$, so for upper
bound purposes we consider only $r$-partite $r$-graphs.) Our $r$-graphs $G$
will have order~$rn$, with vertex set $V=V_1\cup V_2\cup\cdots\cup V_r$,
the $V_i$'s being disjoint sets of size~$n$. Each edge of $G$ has exactly
one vertex in each~$V_i$.

\subsection{Properties of $r$-partite $r$-graphs}\label{subsec:rpart}

A simple random argument, mimicking Erd\H{o}s's work on
Property~B~\cite{E1}, shows that if $G$ is such an $r$-partite $r$-graph
then $\chi_l(G)\le \log_r n+2$. (Suppose $|L(v)|=\ell$ for
all~$v$. Throughout the paper we use the word {\em palette} for the set
$\bigcup_{v\in V(G)}L(v)$ (or a superset of it); it is a set containing all
colours in all lists. For each colour in the palette, select some $V_i$ at
random, and forbid the colour to be chosen by any vertex in $V_i$. Then the
expected number of vertices~$v$ having every colour in $L(v)$ forbidden is
$rnr^{-\ell}$. So if $rnr^{-\ell}<1$ then $G$ is $L$-chooseable.) This
bound holds even for complete $r$-partite $r$-graphs --- that is, where
every possible edge is present. If $r\ge3$, these $r$-graphs are
not simple. But it is not difficult to construct a simple $d$-regular
$r$-partite $r$-graph $G$ with $n$ not much larger than~$d$, thereby giving
examples of simple $d$-regular $r$-graphs with $\chi_l(G)\le (1+o(1))\log_r
d$.

It follows from these remarks and from Proposition~\ref{cor:chil} that the
minimum value of $\chi_l(G)$ amongst simple $d$-regular $r$-graphs lies
between $(1/(r-1)+o(1))\log_rd$ and $(1+o(1))\log_r d$. In the light of the
case $r=2$, one might expect the minimum to be attained by $r$-partite
$r$-graphs of order $rn$ with $n$ close to~$d$, or by random $r$-partite
graphs, and so these are the objects we study.

An important definition is the following. Given an $r$-partite $r$-graph as
just described, and a subset $X\subset V$, let $X_i=X\cap V_i$. We define
$i_X$ to be the index $i$ such that $|X_i|$ is largest. For the sake of
definiteness, if there is more than one such index we take $i_X$ to be the
smallest, though any choice would do. Thus we define
$$
i_X = \min \{i\,: |X_i| = \max \{|X_j|: 1\le j\le r\}\,\}\,.
$$

We can now define the two properties of $G$ that will matter to us. Both
properties involve a condition on sets $X$ stated in terms of the product
of all $|X_i|$ except the largest, that is, a product of $r-1$ quantities.
The first condition is about independent sets, meaning sets~$X$ that
contain no edge of~$G$. The second is about degenerate sets: as usual, we
say that $X$ is $k$-{\em degenerate} if, for every non-empty $Y\subset X$,
the subgraph $G[Y]$ has a vertex of degree at most~$k$. Degenerate sets are
relevant here because they are easily coloured, as noted
in Lemma~\ref{lem:degen}.

\begin{defn}\label{defn:ND}
Let $G$ be an $r$-uniform $r$-partite hypergraph as just described. Let
$d$ be a real number with $1\le d\le n^{r-1}$.
\begin{itemize}
\item $G$ has {\em property} $I(r,n,d)$ if every independent set $X$
satisfies
\begin{equation}
\prod_{i\ne i_X}|X_i|< {n^{r-1}\over d} \log^2d\,.\label{eqn:I}
\end{equation}
\item $G$ has {\em property} $D(r,n,d)$ if every set $X$ that satisfies
\begin{equation}
\prod_{i\ne i_X}|X_i|<{n^{r-1}\over d} \label{eqn:D}
\end{equation}
is $4(\log d/ \log\log d)$-degenerate.
\end{itemize}
\end{defn}

The properties are useful when the parameter $d$ is equal, or near to, the
average degree, though it is convenient not to make this a requirement. In
particular, note that if $G$ has property $D(r,n,d)$ then it has
$D(r,n,d')$ for every $d'>d$; indeed every $G$ has property
$D(r,n,n^{r-1})$ since the only sets then satisfying (\ref{eqn:D}) have
$X_i=0$ for some~$i$. Similarly, if $G$ has property $I(r,n,d)$ then it has
$I(r,n,d')$ for $e^2<d'<d$. The interesting values of $d$ are those for
which $G$ has both $D(r,n,d)$ and $I(r,n,d)$. The apparently strange
dependence on $d$ in the definitions is not crucial to our main theorem: we
need only an expression close to $n^{r-1}/d$ in each of (\ref{eqn:I})
and~(\ref{eqn:D}).  The definitions are stated in the way they are in order
to comply, rather loosely, with properties of random hypergraphs when $d$
is the expected average degree.

\begin{thm}\label{thm:randg}
  There is a number $d_0=d_0(r)$ such that the following holds.  Let $G\in
  \mathcal{G}(n,r,p)$ be a random $r$-partite $r$-uniform hypergraph where
  $p=p(n)$, and let $d=pn^{r-1}\ge d_0$. Then $G$ almost surely has
  properties $I(r,n,d)$ and $D(r,n,d)$.
\end{thm}

This theorem is entirely routine: the point of it is that it gives examples
of $r$-graphs having both properties.  As we shall explain shortly,
$\chi_l(G)$ can be determined very precisely for any $r$-graph $G$ having
both properties, and we can then compare this value with the lower bound
given by Proposition~\ref{cor:chil}. 

The lower bound in Proposition~\ref{cor:chil} is better for
regular~$G$. Random $r$-graphs are close to regular but not quite
regular. In fact the lower bound given for regular $r$-graphs holds for
such close-to-regular graphs too, but it is worth noting the existence of
regular $r$-graphs having the two properties.

\begin{thm}\label{thm:regg}
  There is a number $d_0=d_0(r)$ such that the following holds. Let $d$ be
  an integer with $d\ge d_0$ and let $n\ge r^5d^4$. Then there is a simple
  $d$-regular $r$-partite $r$-graph $G$ having properties $I(r,n,d)$ and
  $D(r,n,d)$.
\end{thm}

\subsection{List chromatic numbers}\label{subsec:listcc}

Our main result is that the list chromatic number of hypergraphs satisfying
both properties can be determined more or less exactly. It will be
expressed in terms of the function $g(r, \alpha)$, a function defined via
what we call {\em preference orders}. The precise definition is delayed
to~\S\ref{sec:preford} because it needs a little discussion.

The parameter $\alpha$, however, can be explained now: it will always be
true that $\alpha=\log_n d$, where $d$ is as in
Definition~\ref{defn:ND}. We specified $d\ge 1$ in that definition so that
$\alpha$ is well-defined. Since $1\le d\le n^{r-1}$ we always have $0\le
\alpha\le r-1$. Notice that, if $G$ is simple, then the average degree is
at most~$n$, and so (since we are imagining $d$ in the definition to be the
average degree), simple hypergraphs are associated with the range
$0\le\alpha\le 1$. Similarly, complete $r$-partite hypergraphs are
associated with $\alpha=r-1$.

Here, at last, is the main theorem.

\begin{thm}\label{thm:mainthm}
Let $G$ be an $r$-uniform $r$-partite hypergraph that satisfies the properties
$I(r,n,d)$ and $D(r,n,d)$ of Definition~\ref{defn:ND}. Then
$$
\chi_l(G)=(g(r,\alpha)+o(1))\log_r d\,.
$$
Here, $\alpha= \log_n d$, the function $g(r,\alpha)$ is
described in terms of preference orders by Definition~\ref{defn:g}, and
the $o(1)$ term tends to zero as $d\to\infty$.
\end{thm}

The theorem is a bit opaque without any information about the function
$g(r,\alpha)$, so we describe some of its properties immediately. It is, in
fact, quite a straightforward function: in particular, for $r=2$ and $r=3$
it is constant, and for every $r$ it is constant over the range
$0\le\alpha\le 1$ associated with simple hypergraphs. Moreover
in~\S\ref{sec:pref} we give an explicit formula for what we believe is the
exact value of $g(r,\alpha)$, though we have no proof.

\begin{thm}\label{thm:elemg} For each $r\in\mathbb{N}$, $r\ge2$, the function
  $g(r,\alpha)$ maps $[0,r-1]$ to $[0,1]$ as follows:
\begin{itemize}
\item[(a)] $g(r,\alpha)$ is continuous and decreasing (that is,
  non-increasing) in $\alpha$,
\item[(b)] $g(r,r-1)=1/(r-1)$,
\item[(c)] $g(2,\alpha)=1$ for $0\le\alpha\le 1$ and $g(3,\alpha)=1/2$ for 
  $0\le \alpha\le 2$,
\item[(d)] for $r\ge 4$, $g(r,\alpha)$ is constant for $0\le \alpha\le
  1+1/(r+3)$,
\item[(e)] $g(4,0)=0.3807\ldots$, and
\item[(f)] $g(r,0)\sim (\log r)/r$ as $r\to\infty$.
\end{itemize}
\end{thm}

The fact that $g(2,\alpha)=1$ means the case $r=2$ of Theorem~\ref{thm:mainthm}
is the theorem of Alon and Krivelevich~\cite{AKr}, though without an
explicit bound on the error term.

Likewise, the fact that $g(r,r-1)=1/(r-1)$ means that, if $G$ is complete
(and we take $d=n^{r-1}$) then $\chi_l(G)=(1/(r-1)+o(1))\log_r
d=(1+o(1))\log_r n$, as noted at the outset of~\S\ref{subsec:rpart}.

As mentioned earlier, our motivation is to investigate whether the lower
bound on $\chi_l$ supplied by Proposition~\ref{cor:chil} for regular simple
$r$-graphs, namely $(1/(r-1)+o(1))\log_r d$, is tight. We also suggested
that, amongst all simple regular $r$-graphs,  ``random-like'' $r$-partite
ones would likely have lowest list-chromatic number. In the light of
Theorem~\ref{thm:randg}, most such $r$-graphs enjoy properties $I(r,n,d)$
and $D(r,n,d)$, so their list-chromatic number is given by
Theorem~\ref{thm:mainthm}. Now $0\le\alpha\le1$ for simple $r$-graphs, and
$g(r,\alpha)$ is constant in this range, so the question of whether the
above approach shows Proposition~\ref{cor:chil} to be tight now comes down
to the question of whether $g(r,0)=1/(r-1)$.

As can be seen from Theorem~\ref{thm:elemg}, $g(r,0)=1/(r-1)$ indeed holds
for $r=2$ and $r=3$, and so Proposition~\ref{cor:chil} is tight in these
cases. For $r\ge4$ we have been unable to determine the exact value of
$g(r,0)$, but we can prove that $g(r,0)>1/(r-1)$. Hence the bound in
Proposition~\ref{cor:chil} appears not to be tight --- indeed, we think it
more likely that the lower bound $(g(r,0)+o(1))\log_r d$ might hold in
general for all simple $r$-graphs of average degree~$d$.

It turns out that the reason why a gap emerges between
Theorem~\ref{thm:regg} and Proposition~\ref{cor:chil} only for $r\ge4$ is
that, for $r=2$, preference orders are more or less trivial, and even
for $r=3$ optimal preference orders are tightly constrained. It is only
when $r\ge4$ that there is room for more interesting preference orders to
exist; more detail appears in~\S\ref{sec:pref}.

As mentioned, we think that $(g(r,0)+o(1))\log_r d$ might be a lower bound
on $\chi_l(G)$ for every $r$-uniform simple hypergraph~$G$ of average
degree~$d$, and to prove this it would be enough to do it for $r$-partite
graphs. In order to obtain a lower bound it is necessary to show that there
is a list function $L:V(G)\to\mathcal{P}(\mathbb{N})$ with
$|L(v)|=(g(r,0)+o(1))\log_r d$ for all $v\in V(G)$, such that $G$ is not
$L$-chooseable. In practice the best lists for this job appear to be random
lists, such as in the proof of Theorem~\ref{thm:lb}, where the bound is
proved for $r$-graphs having property $D(r,n,d)$. We don't have such a
proof for all $r$-graphs, but we can prove a complementary result, namely,
that for any $d$-regular $r$-partite $r$-graph~$G$, if random lists of size
larger than $g(r,0)\log_r d$ are assigned, then $G$ {\em is}
$L$-chooseable. (It is necessary to impose a weak bound on $n$ in terms
of~$d$ for the usual reason that, if we make too many random choices, then bad
things are bound to happen.)

\begin{thm}\label{thm:ranlists}
  Let $\epsilon>0$ and $M>1$ be given. Let $G$ be a simple $d$-regular
  $r$-partite $r$-uniform hypergraph with $n\le d^M$ vertices in each
  class. For each $v\in V(G)$ let a list $L(v)$ of size
  $\ell=\lfloor(1+\epsilon)g(r,0)\log_r d\rfloor$ be chosen uniformly at
  random from a palette of size $t\ge\ell$, independently of other
  choices. Then, with probability tending to one as $d\to\infty$, $G$ is
  $L$-chooseable.
\end{thm}

It is somewhat curious, to us at least, that preference orders are used in
the proof of Theorem~\ref{thm:mainthm} in two entirely different ways, both
in the upper bound (obtained from a colouring algorithm designed around
preference orders --- this is how we first came across them), and also in
the lower bound (for a different reason). This ``coincidence'' is
reminiscent of the relationship with Property~B in the graph case.

As stated earlier, we define preference orders in~\S\ref{sec:preford} and
discuss them enough to be able to define the function~$g(r,\alpha)$. Then,
in~\S\ref{sec:upperbound} we describe the colouring algorithm and prove
Theorems~\ref{thm:ranlists} and~\ref{thm:ub}; the latter theorem is one
half of Theorem~\ref{thm:mainthm}, giving an upper bound for $\chi_l(G)$
when $G$ has property $D(r,n,d)$. A corresponding lower bound, for graphs
with property $I(r,n,d)$, is given by Theorem~\ref{thm:lb}
in~\S\ref{sec:lowerbound}, and this provides the other half of
Theorem~\ref{thm:mainthm}. The elementary probabilistic argument behind
Theorem~\ref{thm:randg} is given in~\S\ref{sec:randg}, and the twist needed
for Theorem~\ref{thm:regg} follows in~\S\ref{sec:regg}. Then,
in~\S\ref{sec:pref}, we examine preference orders in more detail, and
describe how to calculate, or at least to estimate, the
function~$g(r,\alpha)$; we put some effort into this since it is, of
course, at the heart of the paper. Finally in~\S\ref{sec:B} we comment
briefly on the relationship between preference orders and Property~B.

We use standard notation for intervals of real numbers, such as
$[0,1]=\{x\in\mathbb{R}: 0\le x\le 1\}$, and we denote by $[n]$ the set of
integers $\{1,2,\ldots,n\}$.

\section{Preference Orders}\label{sec:preford}

In this section we introduce the notion of preference orders, and define
$g(r,\alpha)$.

To motivate the ideas, consider the most basic case of our problem, where
$G$ is a simple $d$-regular $3$-uniform $3$-graph with $d$ vertices in each
class (that is, $n$=$d$): such a graph is precisely the graph of a Latin
square. As mentioned in~\S\ref{subsec:rpart}, $\chi_l(G)\le\log_3 d+2$, but
this bound holds as well for complete $3$-partite $3$-graphs. For a lower
bound, we have $\chi_l(G)\ge (1/2+o(1))\log_3 d$ from
Proposition~\ref{cor:chil}. The upper bound comes from forbidding each
colour on one of the vertex classes, chosen at random for each colour. To
improve the bound we must allow some colours to appear in {\em every}
class: we call these colours {\em free} and the other colours {\em
  forbidden}. Suppose, for each colour, we make it free with probability
$1-3q$ and otherwise forbid it on one of $V_1$, $V_2$ and $V_3$, with
probability $q$ each. A vertex $v\in V_i$ now chooses a non-free colour
from $L(v)$ if possible (meaning a colour forbidden on some $V_j$, $j\ne
i$), but if there are no such, it chooses a free colour.  Once again, $v$
has no available choice if every colour in $L(v)$ is forbidden on $V_i$,
and we want the expected number of such vertices to be small, say
$3dq^\ell<1/2$. But there is now another potential problem, which is the
presence of monochromatic edges; if each vertex of an edge chooses a free
colour (for each vertex this happens with probability $(1-3q)^\ell$) then
the colours chosen might be the same. The expected number of edges where
each vertex chooses a free colour is at most $d^2(1-3q)^\ell$ (we must
allow for the lists to be overlapping) so we require
$d^2(1-3q)^\ell<1/2$. Taking say $q=0.3028$ and $\ell=0.92\log_3d+2$ makes
both expectations small; hence $\chi_l(G)\le0.92\log_3d+2$.

To get a further improvement, we look for a strategy which will reduce the
likelihood of each vertex in an edge picking the same free colour. For each
of $V_1$, $V_2$ and $V_3$, decide an order of preference on the palette
$\bigcup_{v\in V(G)}L(v)$: denote these orderings by $<_1$, $<_2$
and~$<_3$. The triple $P=(<_1,<_2,<_3)$ is called a {\em preference order}.
Then the choice of $c(v)\in L(v)$ is made as follows: if $v\in V_i$, let
$c(v)$ be a non-free colour in $L(v)$ if one is available, else let $c(v)$
be the most preferred free colour according to the order~$<_i$. We should
design the orderings $<_1$, $<_2$ and~$<_3$ so that a colour preferred in
one class is deprecated in another. A good way to do this is in example
$P_c$ below. In this manner the likelihood of a monochromatic edge is
reduced and, in fact, using the preference order~$P_c$ we obtain
$\chi_l(G)\le0.78\log_3d+3$, as verified in Theorem~\ref{thm:basic}; this
is the best bound we have for Latin square graphs in general, but the
algorithm works only for graphs with a small number of vertices.

We can make further progress if we know something of the structure
of~$G$. We cannot demand that every set of a certain size is independent,
but we can hope to describe sparse sets, and that is what property
$D(r,n,d)$ is doing. For our algorithm to make use of these sparse sets, we
modify it slightly so that $v$ does not commit immediately to the most
preferred free colour in $L(v)$ but, rather, $v$ promises to restrict its
choice to within some small named subset of similarly preferred colours
in~$L(v)$.  If $P$ is well designed then the collection of vertices promising
to use the same subset spans a sparse subgraph, and the colouring can then
be completed (details are in~\S\ref{sec:upperbound}).

What is a good design of preference order~$P$? We assign a value to each~$P$
(Definition~\ref{defn:fP}), and pick the~$P$ of best value: this value is
specifically designed so that the number of vertices choosing a given
colour ties up with the kind of sparse sets guaranteed by property
$D(r,n,d)$.

Are there other ways to use a preference order in a colouring algorithm?
In the simplest conceivable algorithm, each vertex just commits at once to
the most preferred colour in its list. Perhaps surprisingly, such an
algorithm is weak (giving no improvement over $\log_rd$). To make a gain we
need either to use forbidden colours, as we do in Theorem~\ref{thm:basic},
or to incorporate the method of restrictive promises, as we do elsewhere,
using the algorithm set out in detail in~\S\ref{sec:upperbound}. This
algorithm makes no use of forbidden colours; it turns out these give no
extra benefit when restrictive promises are used.

In summary, a preference order is, more or less, a specification of $r$
orders of preference on the palette, one order for each~$V_i$. If the
orderings are all the same then the same colours will be preferred in each
class and a proper colouring is unlikely to be achieved. When $r=2$, and $G$
is a bipartite graph, then, intuitively, one would expect the best palette
order for $V_2$ to be the reverse of that on~$V_1$, and indeed this is
the case --- in fact this method reproduces known results about Property~B
(see~\S\ref{sec:B}). What constitutes a good preference order for $r\ge3$
is what we shall study and, as hinted at before, whereas it is easy to
answer the question for $r=3$, the answer for $r\ge4$ is surprisingly
elusive.

Let us get down to specifics.

\begin{defn}\label{defn:TO}
  Let $<$ be a total ordering of the set $[m]$. Given $k\in[m]$, the {\em
    relative position} ${\rm rpos}_<(k)$ of $k$ in the ordering is $1/m$
  times the number of elements less than or equal to $k$. So ${\rm
    rpos}_<:[m]\to\{1/m,2/m,\ldots,m/m\}$ is a bijection and
  $$
  {\rm rpos}_<^{-1}(1/m)<{\rm rpos}_<^{-1}(2/m) <\cdots<{\rm rpos}_<^{-1}(1)\,.
  $$
\end{defn}

\begin{defn}\label{defn:PO}
  An $(r,m)$-{\em preference order} is an $r$-tuple $P=(<_1, \ldots,<_r)$
  where $<_i$ is a total ordering of~$[m]$, $1\le i\le r$. Abusing
  notation, we write $x\in P$ if $x\in(0,1]^r$ and there is some $k\in [m]$
  such that $x=({\rm rpos}_{<_1}(k),\ldots,{\rm rpos}_{<_r}(k))$.
\end{defn}

Thus $x\in P$ means $x$ is the tuple of relative positions of some element
of~$[m]$. Notice that $\{x:x\in P\}$ determines $P$ to within a permutation
of~$[m]$, because each $x\in P$ tells us the relative position in each
order of some element $k\in[m]$, but we do not know which element. Since
the actual labels of the elements in the ground set $[m]$ are usually
unimportant (for example, when using $P$ in the algorithm above we
generally begin by randomly mapping the palette to~$[m]$), we often think
of the set $\{x:x\in P\}$ as specifying~$P$.

Here are three examples of preference orders. The {\em identity} ordering
is the ordering $1<2<3<\cdots< m$.

\begin{itemize}
\item[$P_a$ ] Let $r=2$, let $<_1$ be the identity ordering, and let $<_2$
  be the reverse of $<_1$; that is, $m<_2(m-1)<_2\cdots<_2 1$. Then
  $$
  \{x:x\in P_a\}=\{(k/m, 1+1/m-k/m): k\in [m]\}\,.
  $$
\item[$P_b$ ] Let $r=3$ and let $m=3p$ be a multiple of three. Let $<_1$ be
  the identity ordering and let $<_2$, $<_3$ be ``rotations'' of
  $<_1$ by $p$ and by $2p$ elements, meaning that
  \begin{eqnarray*}
    2p+1\,<_2\cdots<_2\,3p\,<_2\,1\,<_2\cdots<_2\,p\,
    <_2\,p+1\,<_2\cdots<_2\,2p\\
    p+1\,<_3\cdots<_3\,2p\,<_3\,2p+1\,<_3\cdots<_3\,3p\,
    <_3\,1\,<_3\cdots<_3\,p
    \,.
  \end{eqnarray*}
  Then
  \begin{eqnarray*}
    \{x:x\in P_b\}&=&\{(i/m,1/3+i/m,2/3+i/m): i\in[p]\}\\
    &&\cup\ \, \{(1/3+i/m,2/3+i/m,i/m): i\in[p]\}\\
    &&\cup\ \, \{(2/3+i/m,i/m,1/3+i/m): i\in[p]\}\,.
  \end{eqnarray*}
\item[$P_c$ ] This is the same as $P_b$ except that, in each of $<_1$,
  $<_2$ and $<_3$ we reverse the order of bottom third of the elements,
  that is, we reverse the order of those elements with relative positions
  $1/m$ to $p/m$.  So
  \begin{eqnarray*}
    p\,<_1\cdots<_1\,1\,<_1\,p+1\,<_1\cdots<_1\,2p\,
    <_1\,2p+1\,<_1\cdots<_1\,3p\\
    3p\,<_2\cdots<_2\,2p+1\,<_2\,1\,<_2\cdots<_2\,p\,
    <_2\,p+1\,<_2\cdots<_2\,2p\\
    2p\,<_3\cdots<_3\,p+1\,<_3\,2p+1\,<_3\cdots<_3\,3p\,
    <_3\,1\,<_3\cdots<_3\,p
    \,,
  \end{eqnarray*}
  and
  \begin{eqnarray*}
    \{x:x\in P_c\}&=&\{(1/3+1/m-i/m,1/3+i/m,2/3+i/m): i\in[p]\}\\
    &&\cup\ \, \{(1/3+i/m,2/3+i/m,1/3+1/m-i/m): i\in[p]\}\\
    &&\cup\ \, \{(2/3+i/m,1/3+1/m-i/m,1/3+i/m): i\in[p]\}\,.
  \end{eqnarray*}
\end{itemize}

It turns out that $P_c$ is an essentially optimal choice of preference
order when $r=3$. The next definition defines a parameter of a preference
order, designed to measure its effectiveness in our colouring
algorithm. The parameter captures the way the algorithm makes use of
various independent sets. The form of the definition reflects the
properties of sparse sets in the $r$-graphs we are interested in, set
out in properties $I(r,n,d)$ and $D(r,n,d)$. This is explained in a little
more detail just before Theorem~\ref{thm:ub}.

Analogously to the definition of $i_X$ for a set $X\subset V(G)$, we
define, for an $r$-tuple $x=(x_1,x_2,\ldots,x_r)\in[0,1]^r$
$$
i_x = \min \{i\,: x_i = \max \{x_j: 1\le j\le r\}\,\}\,,
$$
that is, $i_x$ is a specific index of a largest $x_j$.

\begin{defn}\label{defn:fP}
Let $P$ be an $(r,m)$-preference order. Let $0\le \theta\le 1/r$. Then
$$
f_P(\theta)=\max\,\Bigl\{\,\prod_{i\ne  i_x}x_i:\, x\in[\theta,1]^r,\, x\in
  P\,\Bigr\}\,.
$$
\end{defn}

Observe that $f_P(\theta)$ depends only on $\{x:x\in P\}$, supporting the
earlier remark that it is this set that matters rather than~$P$ itself.
Observe too that the set in the definition is non-empty, because there are
fewer than $m/r$ numbers $k\in[m]$ with ${\rm rpos}_{<_1}(k)<1/r$, and
likewise for $<_2,\ldots,<_r$, so there is some $k$ with ${\rm
  rpos}_{<_i}(k)\ge1/r\ge\theta$ for all~$i$. That is, there is some
$x\in[1/r,1]^r$ with $x\in P$. In particular, $f_P(\theta)\ge (1/r)^{r-1}$.

Notice that, by definition, $f_P(\theta)$ is non-increasing in~$\theta$.
The value we are mostly interested in is $f_P(0)$, the maximum of
$\prod_{i\ne i_x}x_i$ over all $x\in[0,1]^r$. This value, when $\theta=0$,
relates to the case $0\le\alpha\le1$ in Theorem~\ref{thm:mainthm}. The
reader who wishes, from now on, to consider only $\theta=0$ will not miss
out on anything of substance.

It is necessary to allow larger $\theta$ in order to handle
larger~$\alpha$. Somewhat vaguely, this is because as $\alpha$ increases
to~$r-1$, meaning $d$ increases to $n^{r-1}$, then the range narrows of those
$x\in P$ that play an interesting role, and $\theta$ captures this
reduced range. For more, we refer to the proofs of Theorems~\ref{thm:ub}
and~\ref{thm:lb}.

Consider the three examples $P_a$, $P_b$ and $P_c$ above. For $x\in P_a$ we
have $x=(x_1,x_2)=(k/m,1+1/m-k/m)$ for some $k\in[m]$. Then
$x_1+x_2=1+1/m$, so one of $x_1$, $x_2$ is at most $1/2+1/2m$ and the other
is at least $1/2+1/2m$. Thus $i_x$ is the index of the larger co-ordinate
and $\prod_{i\ne i_x}x_i=l/m$ for some $l\le (m+1)/2$. Therefore
$f_{P_a}(0)=1/2$ if $m$ is even and $f_{P_a}(0)=1/2+1/2m$ if $m$ is
odd. Moreover it can be seen that $f_{P_a}(\theta)=f_{P_a}(0)$ for
$0\le \theta\le 1/r=1/2$, since the maximum value of $\prod_{i\ne
  i_x}x_i=l/m$ is always attained by some $x$ with $x\in[1/2,1]^2$.

For $x=(x_1,x_2,x_3)\in P_b$ it can be seen that one co-ordinate exceeds
$2/3$ and the other two are $i/m$ and $1/3+i/m$ for some $i\le p=m/3$. Thus
$\prod_{i\ne i_x}x_i=(i/m)(1/3+i/m)$ and $f_{P_b}(0)=2/9$. The maximum is
achieved by some $x$ with $\min x_i\ge 1/3$ and so, once again,
$f_{P_b}(\theta)=f_{P_b}(0)$ for $0\le \theta\le 1/r=1/3$.

For $x=(x_1,x_2,x_3)\in P_c$, one co-ordinate exceeds $2/3$ and the other
two are $1/3+1/m-i/m$ and $1/3+i/m$ for some $i\le p=m/3$. Thus $\prod_{i\ne
  i_x}x_i=(1/3+1/m-i/m)(1/3+i/m)$ and $f_{P_c}(0)=1/9+1/3m$. The maximum is
achieved by some $x$ with $\min x_i\ge 1/3$ and so, once again,
$f_{P_c}(\theta)=f_{P_c}(0)$ for $0\le \theta\le 1/r=1/3$.

In the three examples, $f_P(\theta)$ is constant for $0\le\theta\le
1/r$. This is a reflection of the fact, noted in~\S\ref{subsec:listcc},
that the overall situation is more straightforward for $r\le 3$ and new
phenomena appear only when $r\ge4$.

It turns out that the best preference orders for the colouring algorithm
are those with the lowest values of $f_P$. This leads us to the next
definition.

\begin{defn}\label{defn:fg}
Let $r\ge2$. For $0\le\theta\le 1/r$ we define
\begin{eqnarray*}
f(r,\theta,m)\,&=&\,\min\,\{\,f_P(\theta):\, \mbox{$P$ is an $(r,m)$-preference
  order}\}\\
\mbox{and\quad}f(r,\theta)\,&=&\,\inf\,\{\,f(r,\theta,m):\, m\in\mathbb{N}\}\,.
\end{eqnarray*}
\end{defn}

It was noted that $f_P(\theta)$ is non-increasing in~$\theta$, and hence so
are $f(r,\theta,m)$ and $f(r,\theta)$. Moreover we saw that
$f_P(\theta)\ge(1/r)^{r-1}$ for all $P$
and~$\theta$, so $f(r,\theta)\ge(1/r)^{r-1}$ for all~$\theta$. The examples
$P_a$ and $P_c$ show that $f(2,\theta)\le 1/2$ for $0\le \theta\le 1/2$ and
$f(3,\theta)\le 1/9$ for $0\le \theta\le 1/3$. Hence equality holds in each
of these cases. In particular, when $r=2,3$, then $f(r,\theta)$ is constant
for $0\le \theta\le 1/r$.

We are, at last, in a position to define $g(r,\alpha)$. To do this, we need
to relate a value of $\theta$ to each~$\alpha$. Formally,
this special value is
$\beta(\alpha)=\sup\{\theta:\theta^\alpha\le f(r,\theta)\}$, which exists
for $\alpha>0$. But it follows from simple
properties of 
$f(r,\theta)$, given below in Theorem~\ref{thm:elem}\,(a)(b), that
$\beta(\alpha)$ is the unique solution
to $\theta^\alpha=f(r,\theta)$. So,
anticipating those properties, we take the simpler statement as the definition.

\begin{defn}\label{defn:g}
Let $r\ge2$ and $0<\alpha\le r-1$. Define
$\beta=\beta(\alpha)$ by $\beta^\alpha=f(r,\beta)$.
Then we define $g(r,\alpha)=-1/\log_r(f(r,\beta))$. Note that if
$\alpha=\log_n d$ then
\begin{equation}
f(r,\beta)^{g(r,\alpha)}\,=\,{1\over r}\,,
\qquad
f(r,\beta)^{g(r,\alpha)\log_r d}\,=\,{1\over d}
\qquad\mbox{and}\qquad
\beta^{g(r,\alpha)\log_r d}\,=\,{1\over n}
\,.\label{eqn:fng}
\end{equation}
\end{defn}

Observe that $g(r,0)$ is not defined by this statement but, since
$g(r,\alpha)$ is constant for $0<\alpha\le1$ (see Theorem~\ref{thm:elemg})
then we define $g(r,0)$ to equal this constant value.

We remark that $\beta(r-1)=1/r$ because $f(r,1/r)=(1/r)^{r-1}$
(Theorem~\ref{thm:elem}\,(b)). Moreover $\beta(\alpha)$ is strictly 
increasing: for if $\alpha_1<\alpha_2$ and
$\beta(\alpha_1)\ge\beta(\alpha_2)$, then
$f(r,\beta(\alpha_1))=\beta(\alpha_1)^{\alpha_1}>\beta(\alpha_1)^{\alpha_2}
\ge\beta(\alpha_2)^{\alpha_2}=f(r,\beta(\alpha_2))$, contradicting the fact
that $f(r,\theta)$ is decreasing (Theorem~\ref{thm:elem}\,(a)).

Notice how the expression $g(r,\alpha)\log_r d$, appearing in
Theorem~\ref{thm:mainthm}, appears also in~(\ref{eqn:fng}). In the proof of
the theorem, we try to appeal to~(\ref{eqn:fng}) directly rather than to
the definition of $g(r,\alpha)$.

The next theorem lists some basic properties of $f(r,\theta)$, in the same
way that Theorem~\ref{thm:elemg} lists some of those of $g(r,\alpha)$. In
particular it shows that $f(r,\theta)$ is constant for small~$\theta$,
which is the reason $g(r,\alpha)$ is constant for small $\alpha$.

\begin{thm}\label{thm:elem}
  For each $r\in\mathbb{N}$, $r\ge2$, the function $f(r,\theta)$ maps
  $[0,1/r]$ to $[0,1]$ as follows:
  \begin{itemize}
  \item[(a)] $f(r,\theta)$ is continuous and decreasing in $\theta$,
  \item[(b)] $f(r,1/r)= (1/r)^{r-1}$,
  \item[(c)] for $r>2$, $f(r,\theta)\le f(r-1,\theta)$,
  \item[(d)] $f(2,\theta)=1/2$ for $0\le\theta\le1/2$ and
    $f(3,\theta)=1/9$ for $0\le\theta\le1/3$, 
  \item[(e)] for $r\ge 4$, $f(r,\theta)$ is constant for
    $0\le\theta\le (1-1/r)e^{-r+1}$, 
  \item[(f)] $f(4,0)=0.0262\ldots$, and
  \item[(g)] $((r-1)/er)^{r-1}\le f(r,0)\le (r-1)!/r^{r-1}$.
  \end{itemize}
\end{thm}

Theorems~\ref{thm:elemg} and~\ref{thm:elem} are proved in~\S\ref{sec:pref}.

\section{A list colouring algorithm and some upper 
bounds}\label{sec:upperbound}

In order to prove $\chi_l(G)\le \ell$ for some $\ell$, we need an algorithm
that will colour $G$ whenever the vertices are given lists of
$\ell$~colours each. 

We start with a proof of an upper bound for Latin square graphs, mentioned
earlier in~\S\ref{sec:preford}, which uses preference orders in an
elementary way. The proof makes no use of the structure of the graph and
does not, in fact, require simplicity. It does make use of randomization.

\begin{thm}\label{thm:basic}
  Let $G$ be a $d$-regular 3-partite 3-graph with vertex classes of
  size~$d$. Then $\chi_l(G)\le 0.78\log_3d+3$.
\end{thm}
\begin{proof}
  Let $\ell=\lceil 0.78\log_3d\rceil+2$ and assume each vertex $v$ has a
  list $L(v)$ of $\ell$ colours to choose from. Let $m$ be the size of the
  palette $\bigcup_{v\in V(G)}L(v)$; by increasing $m$ if need be, we can
  assume $m$ is divisible by~$3$. Take a random map $\Phi:\bigcup_{v\in
    V(G)}L(v) \to[m]$, and let $P_c$ be the $(3,m)$-preference order given
  as an example in~\S\ref{sec:preford}.

  Let $q^2=(1-2q)/9$, so $q=(-1+\sqrt{10})/9\approx0.24$. As described
  in~\S\ref{sec:preford}, each colour in the palette is forbidden on one of
  $V_1$, $V_2$ or $V_3$, with probability $q$ each, and is otherwise free,
  with probability $1-3q$. If $v\in V_i$ then $c(v)$ is taken to be a
  non-free colour, if $L(v)$ has one available, else it is the free colour
  whose image under $\Phi$ is most preferred in the ordering~$<_i$.

  There are two ways the colouring can fail: a vertex might have no colour
  available, or an edge might be monochromatic. The expected number of
  vertices with no colours available, that is, all colours in $L(v)$ are
  forbidden on $V_i$, is $3dq^\ell<1/2$. Suppose now some edge
  $e=\{v_i,v_2,v_3\}$ is monochromatic, where $v_i\in V_i$, $1\le i\le 3$:
  say $c(v_1)=c(v_2)=c(v_3)=\gamma$. Then $\gamma$ must be free. Observe
  that any other colour lying in more than one of $L(v_1)$, $L(v_2)$ and
  $L(v_3)$ is free, else it would have been chosen by one of the
  vertices. Let $\Phi(\gamma)=k$ and let $p_i={\rm rpos}_{<_i}(k)$, $1\le
  i\le3$; for ease of notation assume $k\le m/3$ so $p_1\le p_2\le p_3$. If
  $\gamma'\in L(v_1)$ is free and $\gamma'\ne\gamma$ then
  $\Phi(\gamma')<_1\Phi(\gamma)$. By the definition of $P_c$ this means
  $\Phi(\gamma)<_2\Phi(\gamma')$ and $\Phi(\gamma)<_3\Phi(\gamma')$, so
  $\gamma'\notin L(v_2)\cup L(v_3)$. Thus $L(v_1)\cap L(v_i)=\{\gamma\}$
  for $i=2,3$; let $j=|L(v_2)\cap L(v_3)|$.
  
  Consider the event $M_e$ that $e$ is monochromatic (necessarily of
  colour~$\gamma$, given what we now know of $L(v_i)$). Let $p_i'=p_i-1/m$,
  $i=1,2$. The probability that $c(v_1)=\gamma$ is at most
  $(q+(1-3q)p_1')^{\ell-1}$, because every colour in
  $L(v_1)\setminus\{\gamma\}$ must either be forbidden on~$V_1$ or must map
  under $\Phi$ to a relative position below~$p_1$. Treating in like manner
  $(L(v_2)\cap L(v_3))\setminus\{\gamma\}$, $L(v_2)\setminus L(v_3)$ and
  $L(v_3)\setminus L(v_2)$, we have
  $$
  \Pr(M_e)\le
  (q+(1-3q)p'_1)^{\ell-1}((1-3q)p'_2)^{j-1}(q+(1-3q)p'_2)^{\ell-j}
  (q+(1-3q))^{\ell-j}\,.
  $$
  Since $p_2'\le 2/3$ and $q\approx1/4$ we have $q>(1-3q)p_2'$, so
  $(1-3q)p_2'\le (q+(1-3q)p_2')/2<(q+(1-3q)p_2')(1-2q)$. Hence the bound for 
  $ \Pr(M_e)$ decreases with~$j$, and so 
  $$
  \Pr(M_e)\le[(q+(1-3q)p_1')(q+(1-3q)p'_2))(1-2q)]^{\ell-1}\,.
  $$
  But $p_1'=1/3-x-1/m\le 1/3-x$ and $p_2'=1/3+x$ for some $x\ge0$, so
  $(q+(1-3q)p_1')(q+(1-3q)p'_2)\le 1/9$. Thus $\Pr(M_e)\le
  ((1-2q)/9)^{\ell-1} = q^{2\ell-2}$.

  Finally, there are $d^2$ edges in $G$, so the expected number of monochromatic
  edges is at most $d^2q^{2\ell-2}<1/2$. Hence there is some mapping $\Phi$
  for which every vertex has a choice of colour and for which no edge is
  monochromatic, proving the theorem.
\end{proof}

As discussed in~\S\ref{sec:preford}, the algorithm used in
Theorem~\ref{thm:basic} is too weak for general use, and we turn now to the
main algorithm. It too uses randomized preference orders.

\medskip
\paragraph*{{\bf Algorithm} for list colouring an $r$-partite $r$-graph~$G$
having lists of size $\ell$}
\begin{itemize}
\item Let $[t]$ be the palette. Choose parameters $k$ and $\delta$. Let
  $m=\delta \ell/k$.
\item Randomly partition the palette into $m$ blocks $B_1,\ldots,B_m$ of
  equal size (increase $t$ if need be). Choose an $(r,m)$-preference order
  $P=(<_1,\ldots,<_r)$.
\item Let $\mathcal{B}=\{B_1,\ldots,B_m\}$. Say $B\in\mathcal{B}$ is {\em
    available} to $v\in V(G)$ if $|L(v)\cap B|>k$. 
\item Define $b:V(G)\to\mathcal{B}$ by $b(v)=B_q$ where, if $v\in V_i$,
  then $q$ is the member of $\{j:B_j \mbox{ is available to $v$}\}$ of
  greatest relative position in the order~$<_i$.
\item For $B\in\mathcal{B}$ let $X(B)=\{v:b(v)=B\}$. Colour $G[X(B)]$ using
  colours from~$B$.
\end{itemize}

We shall choose $\delta<1$ small. Since, for $v\in V(G)$, at most
$mk=\delta \ell$ colours are in blocks unavailable to $v$, there are at
least $(1-\delta)\ell$ colours in $L(v)$ in available blocks: in particular
$b(v)$ is well-defined. In effect, $v$ is promising to choose a colour
$c(v)$ from the block~$b(v)$, this block being the most preferred amongst
blocks available to~$v$ (where $v\in V_i$ uses the order~$<_i$). The
algorithm will succeed --- that is, it will show $G$ is $L$-chooseable, if
for each $B\in\mathcal{B}$ we can colour $G[X(B)]$ using colours from~$B$,
because the sets $X(B)$ partition $V(G)$ and the sets $B$ partition $[t]$.

Since $|L(v)\cap B|>k$ for each $v\in X(B)$, the algorithm will succeed if
the subgraph $G[X(B)]$ is $k$-degenerate, as verified by applying the next
(standard and elementary) lemma to $H=G[X(B)]$.

\begin{lem}\label{lem:degen}
Let $H$ be a $k$-degenerate $r$-graph. Let
$L:V(H)\to\mathcal{P}(\mathbb{N})$ be a list assignment with
$|L(v)|>k$ for every vertex~$v$. Then $H$ is $L$-chooseable.
\end{lem}
\begin{proof} 
  Construct an ordering $v_1,\ldots,v_n$ of the vertices of $H$ in which
  $v_j$ has minimum degree in the subgraph $H[\{v_1,\ldots,v_j\}]$, $1\le
  j\le n$. Now, for $j=1,\ldots,n$ in turn, choose a colour $c(v_j)\in
  L(v_j)$ as follows. There are at most $k$ edges in
  $H[\{v_1,\ldots,v_j\}]$ that contain~$v_j$: select a vertex other than
  $v_j$ in each of these edges, and then choose $c(v_j)\in L(v_j)$
  different from the colours of the selected vertices (possible since
  $|L(v_j)|>k$). The resultant colouring is a proper colouring of~$H$.
\end{proof}

We give two examples of the use of the algorithm. In each case, proving
that the algorithm succeeds amounts to showing that $G[X(B)]$ is
$k$-degenerate, for each block $B\in\mathcal{B}$. The first example
supplies the upper bound for Theorem~\ref{thm:mainthm}, the second example
establishes Theorem~\ref{thm:ranlists}.

Broadly speaking, the first example works for the following reason. There
is an $r$-tuple $x=(x_1,\ldots,x_r)$ of relative positions of the block $B$
in the preference order~$P$. If $x_i$ is small then the number $|X_i|$ of
vertices in $V_i$ for which $b(v)=B$ will, very likely, be correspondingly
small. If $x_j<\theta$ for some~$j$ (where $\theta$ is determined
by~$\alpha$), it turns out that $X_j=\emptyset$, so certainly $G[X]$ is
$k$-degenerate. On the other hand, if $x_j\ge\theta$ for all~$j$, then by
definition of $f_P(\theta)$ we know $\prod_{i\ne i_x}x_i\le
f_P(\theta)$. This leads to a bound on $\prod_{i\ne i_X}|X_i|$ which,
because of property $D(r,n,d)$, again means $G[X]$ is $k$-degenerate. The
second example works in a similar way but finishes differently; because
$\theta=0$ (as $G$ is simple) then $\prod_{i\ne i_X}|X_i|$ must be bounded,
and since $X$ is a random set (as the lists were 
chosen randomly) we again conclude that $G[X]$ is $k$-degenerate.

We give a quantitative bound in the first example, with a rate at which the
$o(1)$ term tends to zero as $d\to\infty$. This bound depends on two
factors, one being the value of $k$ for which the sets in property
$D(r,n,d)$ are $k$-degenerate, and the other being the rate at which
$f(r,\theta,m)\to f(r,\theta)$ as $m\to\infty$. It turns out to be the
second of these that predominates in our analysis; we use a bound on the
rate proved in~\S\ref{sec:pref}.

\begin{thm}\label{thm:ub}
  Let $r\ge2$. Then there exists $d_1=d_1(r)$ such that, if $d>d_1$ and $G$
  is an $r$-uniform $r$-partite hypergraph with property $D(r,n,d)$, then
  $$
  \chi_l(G)\le (g(r,\alpha)+ (\log\log d)^{-1/5})\log_r d\,
  $$
  where $\alpha=\log_n d$.
\end{thm}
\begin{proof}
  All estimates in the proof hold provided $d_1(r)$ is large enough: we
  ignore integer parts. Let lists of $\ell$ colours be assigned to each
  vertex of $G$, where $\ell=(g(r,\alpha)+(\log\log d)^{-1/5})\log_r d$.
  Let $[t]$ be the palette comprising all the colours in all the lists;
  clearly $t\ge \ell$. Define $k=4\log d/ \log\log d$ and $\delta=(\log\log
  d)^{-1/4}$. Further define $m=\delta\ell/k$. By adding a few dummy
  colours to the palette if necessary, we may assume that $t$ is divisible
  by $m$. 

  Let $\beta=\beta(\alpha)$ as specified in Definition~\ref{defn:g}. 
  There is some $(r,m)$-preference order
  $P=(<_1,\ldots,<_r)$ with $f_P(\beta)=f(r,\beta,m)$. Apply the algorithm
  above to~$G$, using $k$, $\delta$ and $P$ as just specified. What remains
  is to show that $G[X(B)]$ is $k$-degenerate for each $B\in\mathcal{B}$.

  Here is the central part of the argument. Consider some particular block
  $B$, and let $X=X(B)$. Let $x=(x_1,\ldots,x_r)$ be the $r$-tuple of
  relative positions of $B$ in the preference order~$P$: that is, if, say,
  $B=B_j$, then $x_i$ is the relative position of $j$ in~$<_i$. Let $v\in
  X_i$. We know that at least $(1-\delta)\ell$ of the colours in $v$'s list
  lie in available blocks, and, by definition of~$X$, these blocks all lie
  in relative positions $x_i$ or below in the $i$'th order. There are
  $x_it$ colours from $[t]$ in blocks $B$ or below it in the $i$th order,
  so the probability that the random partition of $[t]$ into blocks results
  in $(1-\delta)\ell$ of $v$'s colours being placed in these low blocks is
  at most
  $$
  {\ell\choose (1-\delta)\ell}{x_i t\choose (1-\delta)\ell}{t\choose
    (1-\delta)\ell}^{-1}\le {\ell\choose \delta \ell} x_i^{(1-\delta)\ell}
  \le \left({e\over\delta}\right)^{\delta\ell}x_i^{(1-\delta)\ell}\,.
  $$
  Hence, by Markov's inequality, the inequality $|X_i|\le
  rm(e/\delta)^{\delta\ell}x_i^{(1-\delta)\ell}n$ holds with probability
  exceeding $1-1/rm$, and thus, with probability more than $1-1/m$, the
  inequality holds for $1\le i\le r$. Consequently, with positive
  probability, there exists a partition of $[t]$ such that the inequality
  holds for every block~$B$, and for every set $X_i=X(B)\cap V_i$, $1\le
  i\le r$. 

  To finish the proof, it is now enough to check that if
  $x=(x_1,\ldots,x_r)\in P$, and $X\subset V(G)$ satisfies $|X_i|\le
  rm(e/\delta)^{\delta\ell}x_i^{(1-\delta)\ell}n$, then $X$ is
  $k$-degenerate. There are two possibilities: either $\prod_{i\ne
    i_x}x_i\le f_P(\beta)$, or $\prod_{i\ne i_x}x_i> f_P(\beta)$.

  Consider the first possibility, that $\prod_{i\ne i_x}x_i\le
  f_P(\beta)$. Then
  $$
  \prod_{i\ne i_X}|X_i|\le (rm)^r \left({e\over\delta}\right)^{r\delta\ell}
  \left(\prod_{i\ne i_x} x_i\right)^{(1-\delta)\ell}n^{r-1} \le (rm)^r
  \left({e\over\delta}\right)^{r\delta\ell}
  f_P(\beta)^{(1-\delta)\ell}n^{r-1}\,.
  $$
  Now $f_P(\beta)=f(r,\beta,m)\le f(r,\beta)+2^r\sqrt{(\log rm)/m}$ by
  Lemma~\ref{lem:fdecr}. Theorem~\ref{thm:elem} tells us that $f(r,\beta)$
  and $g(r,\alpha)$ are bounded below (namely $f(r,\beta)\ge
  f(r,1/r)=(1/r)^{r-1}$ and $g(r,\alpha)\ge g(r,1/r)=1/(r-1)$) so,
  recalling the definitions of $k$, $\delta$ and~$m$, we have
  $f_P(\beta)\le f(r,\beta)(1+\delta)\le f(r,\beta)e^\delta$. Put
  $\Lambda=(\log_r d)(\log\log d)^{-1/5}$, so $\ell=g(r,\alpha)\log_r d +
  \Lambda$. Then, using~(\ref{eqn:fng}), we obtain
  $$
  \prod_{i\ne i_X}|X_i|\le (rm)^r \left({e\over\delta}\right)^{r\delta\ell}
  f(r,\beta)^{(1-\delta)\ell}e^{\delta\ell}n^{r-1}
  = (rm)^r \left({e\over\delta}\right)^{r\delta\ell}
  f(r,\beta)^{\Lambda-\delta\ell}e^{\delta\ell}{n^{r-1}\over d}\,.
  $$
  By Theorem~\ref{thm:elem}, $f(r,\beta)\le f(2,0)=1/2$, so we conclude
  that $\prod_{i\ne i_X}|X_i|\le Kn^{r-1}/d$, where $K=(rm)^r
  (e/\delta)^{r\delta\ell}e^{\delta\ell}2^{-\Lambda+\delta \ell}$. Since $\Lambda$ is
  much larger than either $\delta\ell\log(1/\delta)$ or $\log m$, we see
  that $K<1$. Hence $\prod_{i\ne i_X}|X_i|< n^{r-1}/d$ and,
  because $G$ has property $D(r,n,d)$, this means $X=X(B)$ is $k$-degenerate,
  so resolving the first of the two possibilites.

  Consider now the second possibility, where $\prod_{i\ne
    i_x}x_i>f_P(\beta)$. By definition of $f_P(\beta)$ there must be some
  index~$j$ with $x_j < \beta$. Therefore, using equation~(\ref{eqn:fng}),
  and the fact that (by definition) $\beta(\alpha)\in [0,1/r]$, we have
  $$
  |X_j|< rm(e/\delta)^{\delta\ell}\beta^{(1-\delta)\ell}n =
  rm(e/\delta)^{\delta\ell}\beta^{\Lambda- \delta
    \ell}\le rm(e/\delta)^{\delta\ell}r^{-\Lambda+ \delta \ell}<K\,,
  $$
  where $\Lambda$ and $K$ are as in the previous paragraph. But we saw that
  $K<1$, and so $|X_j|<1$, meaning $X_j=\emptyset$. But then $X$ contains
  no edges, and so is certainly $k$-degenerate. This resolves the second
  of the two possibilities, completing the proof of the theorem.
\end{proof}

Our second example of the use of the algorithm is a proof of
Theorem~\ref{thm:ranlists}.

\begin{proof}[Proof of Theorem~\ref{thm:ranlists}.]
  Let $G$ and the lists $L(v)$ be as stated. We choose constants $k$ and
  $\delta$ as follows. First, write $\ell=g(r,0)\log_r d + \Lambda$, so
  $\Lambda\approx \epsilon g(r,0)\log_r d$. Then choose $\delta<1$ small
  enough that $(r\ell)^r (e/\delta)^{r\delta\ell}e^{\delta\ell}2^{-\Lambda+\delta
    \ell}<2^{-\Lambda/2}$ (assuming, as we may, that $\ell$ is large). Then
  choose $k$ so that $2^{-k\Lambda/2}< 1/d^{M+1}$.

  As usual, put $m=\delta\ell/k$ and assume $t$ is a multiple of~$m$. Let
  $\alpha=\log_n d$. Since $G$ is simple, $d\le n$; thus $\alpha\le 1$ and
  (by Theorem~\ref{thm:elemg}) $g(r,\alpha)=g(r,0)$. Select an
  $(r,m)$-preference order $P=(<_1,\ldots,<_r)$ with
  $f_P(0)=f(r,0,m)$. Apply the algorithm with $k$, $\delta$ and $P$ as
  specified: we need only show that $G[X(B)]$ is $k$-degenerate for each
  block of colours $B\in\mathcal{B}$. Fix some block $B=B_j$ and, as in the
  proof of Theorem~\ref{thm:ub}, let $x=(x_1,\ldots,x_r)\in P$ be the tuple
  of relative positions of~$j$ in the orders $<_1,\ldots,<_r$.

  The vertex lists are chosen randomly. We can imagine the algorithm first
  makes the random partition of the palette, and afterwards the assignment
  of lists is made to the vertices. The first step determines the
  collections $\mathcal{L}_i$, $1\le i\le r$, of vertex lists such that, if
  $v\in V_i$ and $L(v)\in\mathcal{L}_i$, then $b(v)=B$. The second step
  determines which vertices $v\in V_i$ receive a list from~$\mathcal{L}_i$,
  namely, it determines~$X_i$. Hence we can consider $X_i$ to have been
  generated in the following way: first, its size $|X_i|$ is chosen from a
  binomial distribution with parameters $n,|\mathcal{L}_i|/{t \choose
    \ell}$, and then, having decided the size $|X_i|$, $X_i$ itself is a
  random $|X_i|$-subset of~$V_i$. In fact, having partitioned the palette,
  we may choose the sizes $|X_i|$ for every $B\in\mathcal{B}$ and
  every~$i$, $1\le i\le r$, before choosing the sets $X_i$ themselves. In
  the proof of Theorem~\ref{thm:ub}, we showed if $v\in V_i$ and $v$ has
  some list $L(v)$ then the probability that $b(v)=B$ is at most
  $(e/\delta)^{\delta\ell}x_i^{(1-\delta)\ell}$. But this probability is
  the probability that $L(v)\in\mathcal{L}_i$, and this equals
  $|\mathcal{L}_i|/{t \choose \ell}$; hence $|\mathcal{L}_i|/{t \choose
    \ell}\le (e/\delta)^{\delta\ell}x_i^{(1-\delta)\ell}$. Using Markov's
  inequality again as in the proof of Theorem~\ref{thm:ub}, we may assume
  that all the chosen sizes $|X_i|$ satisfy $|X_i|\le
  rm(e/\delta)^{\delta\ell}x_i^{(1-\delta)\ell}n$.

  We now re-use a calculation performed in the first possibility in the
  proof of Theorem~\ref{thm:ub}, though much less care is needed with the
  estimates this time. Taking $\beta=0$, and noting that $m=\Theta(\log
  d)$, we have once again $f_P(0)\le f(r,0)e^\delta$, and so $\prod_{i\ne
    i_X} (|X_i|/n)\le K/d$, where $K=(rm)^r
  (e/\delta)^{r\delta\ell}e^{\delta\ell}2^{-\Lambda+\delta \ell}$. Since
  $m<\ell$, we have $K<2^{-\Lambda/2}$ by choice of~$\delta$.

  Let $v\in V_{i_X}$ and let $E$ be one of the $d\choose k$ choices of a
  set of $k$ edges containing~$v$. Given that $G$ is simple, the
  probability, conditional on $v\in X$, that the edges in $E$ lie within
  $X$ is $\prod_{i\ne i_X}{n-k\choose |X_i|-k}/{n\choose |X_i|}\le
  \prod_{i\ne i_X} (|X_i|/n)^k\le(K/d)^k$. Thus the probability that the
  degree of $v$ in $G[X]$ exceeds $k$ is at most ${d\choose k}(K/d)^k\le
  K^k<2^{-k\Lambda/2}<1/d^{M+1}$, by choice of~$k$. This probability is
  less than~$1/nd$, so with probability exceeding $1-1/d$, every vertex in
  $X_{i_X}$ has degree at most $k$ in $G[X]$; because $G$ is $r$-partite this
  certainly implies $G[X]$ is $k$-degenerate.

  So, given~$B\in\mathcal{B}$, $G[X(B)]$ is $k$-degenerate with probability
  more than $1-1/d$, and since $|\mathcal{B}|=m=o(d)$ this means that, with
  probability tending to one, $G[X(B)]$ is $k$-degenerate for every
  $B\in\mathcal{B}$ and thus $G$ is $L$-colourable, proving the theorem.
\end{proof}

\section{A lower bound}\label{sec:lowerbound}

To prove the lower bound in Theorem~\ref{thm:mainthm} we shall choose some
lists for $G$ at random. We make use of the following basic tail estimate.

\begin{prop}[{\cite[Theorem~2.1, Theorem~2.8]{JLR}}]\label{prop:chernoff}
  If $Y$ is binomially distributed, with mean $\lambda$,
  then $\mathbb{P}(Y\le \lambda-y)\le e^{-y^2/2\lambda}$. The same bound
  holds for any sum $Y$ of independent Bernoulli variables.
\end{prop}

\begin{rem}\label{rem:chernoff}
  The bound of Proposition~\ref{prop:chernoff} holds if $Y$ is
  hypergeometrically distributed. This can be proved either by a comparison
  of moment generating functions, on which the inequality is based
  (Hoeffding~\cite{Hoef}) or by showing that in this case $Y$ is in fact a
  sum of independent Bernoulli variables (Vatutin and Mikhailov~\cite{VM}
  --- the proof is reproduced in~\cite{HT} and the idea goes back at least to
  Harper~\cite{H}). More generally, the bound
  holds for variables of the form $Y=|X\cap T_1\cdots\cap T_r|$, where
  $X,T_1,\ldots,T_r\subset [n]$, $X$ is fixed and $T_1,\ldots,T_r$ are
  chosen independently and uniformly of fixed sizes $|T_i|=t_i$, $1\le i\le
  r$. When $r=1$ then $Y$ is hypergeometrically distributed: the
  general case can be derived from the generating function proof by
  induction on~$r$, but in fact it is already shown explicitly
  in~\cite[Corollary~5]{VM} that $Y$ of this form are sums of independent
  Bernoulli variables. The authors thank Svante Janson for pointing them
  to~\cite{VM}.
\end{rem}

The next straightforward lemma provides the properties that we need of the
lists. The size of the palette $[t]$ from which the lists are chosen is not
particularly significant.

\begin{lem}\label{lem:lists}
  Let $\ell,n\in\mathbb{N}$, $\ell\ge3$, and let $\zeta\in(0,1]$. Let
  $t=\lceil2\ell^2/\zeta\rceil$. Suppose that $n\zeta^\ell\ge 16t$. Then
  there exists a sequence $\mathcal{L}=(L_i)_{i\in[n]}$ of elements of
  $[t]^{(\ell)}$ such that, for every $Z\subset [t]$ with $|Z|=zt\ge \zeta
  t$, we have $|\{i\in[n]\,:\,L_i \subset Z\}|\ge nz^\ell/4$.
\end{lem}
\begin{proof} 
  For each $i\in[n]$, choose $L_i$ uniformly at random in $[t]^{(\ell)}$,
  independently of other choices. Let $Z\subset [t]$ have size $zt\ge \zeta
  t$. Let $Y=\{i\in[n]\,:\,L_i \subset Z\}$. Then $Y$ is binomially
  distributed with parameters $n, p={|Z|\choose \ell}/{t\choose \ell}$. By
  Proposition~\ref{prop:chernoff}, $\Pr(Y\le np/2)\le \exp(-np/8)$. Now
  $np= n{zt\choose \ell}/{t\choose \ell}=nz^\ell\prod_{i=0}^{\ell-1}
  (t-i/z)/(t-i) \ge nz^\ell(1-\ell/(z(t-\ell)))^\ell\ge
  nz^\ell(1-1/(2\ell-1))^\ell\ge nz^\ell/2$, the penultimate inequality
  following from the fact that $t\ge 2\ell^2/z$ and the last because
  $\ell\ge3$. Thus $\Pr(|Y|\le nz^\ell/4)\le \Pr(Y\le np/2)\le
  \exp(-nz^\ell/16)\le \exp(-t)$. There are at most $2^t$ sets~$Z$, so with
  positive probability $|\{i\in[n]\,:\,L_i \subset Z\}|\ge nz^\ell/4$ holds
  for every $Z\subset[t]$, proving the lemma.
\end{proof}

The next theorem establishes the lower bound in
Theorem~\ref{thm:mainthm}. The argument is roughly this. We assign lists of
colours to the vertices using Lemma~\ref{lem:lists}. Suppose it is possible
to colour the graph. We obtain a preference order on the palette by
letting $<_i$ be the order of popularity of the colours on~$V_i$ in this
colouring. Thus there is some colour (green, say) whose relative positions
$x=(x_1,\ldots,x_r)$ satisfy $x_j>\theta$ for all~$j$ ($\theta$ determined
by~$\alpha$) and $\prod_{i\ne i_x}x_i\ge f(r,\theta)$. By the properties of
the lists this yields a lower bound on $\prod_{i\ne i_X}|X_i|$, where $X$
is the set of vertices choosing green. But this lower bound is incompatible
with $G$ having property $I(r,n,d)$ and the fact that $X$ is independent.

\begin{thm}\label{thm:lb}
  Let $r\ge2$. Then there exists $d_2=d_2(r)$ such that, if $d>d_2$ and $G$
  is an $r$-uniform $r$-partite hypergraph of order $rn$ having property
  $I(r,n,d)$, then
$$
\chi_l(G)>g(r,\alpha)\log_r d- 6r\log\log d\,.
$$
where $\alpha=\log_n d$.
\end{thm}
\begin{proof}
  All estimates hold provided $d_2$ is large enough: we ignore integer
  parts. Let $G$ be a graph as in the theorem.  Let $\ell=g(r,\alpha)\log_r
  d - 6r\log\log d$. Let $\zeta=\max\{\beta(\alpha), (1/r)^{r-1}\}$. Recall
  that $\beta(\alpha)\in[0,1/r]$, and so $(1/r)^{r-1}\le\zeta\le 1/r$.
  Using~(\ref{eqn:fng}), we have
  $$
  n\zeta^{\ell+1}\ge n\beta(\alpha)^{g(r,\alpha)\log_r d}(1/r)^{-6r\log\log d+1} 
  = r^{6r\log\log d-1}\,.
  $$
  Thus $n\zeta^{\ell+1}\ge 2^{6r\log\log d-1}\ge 2^{2\log_2\log_2 d +
    6}=64(\log_2d)^2\ge64\ell^2$, since $g(r,\alpha)\le 1$. Hence
  $n\zeta^\ell\ge 16t$ where $t=\lceil2\ell^2/\zeta\rceil$. So we can
  apply Lemma~\ref{lem:lists} to obtain lists $L_1,\ldots,L_n$ of $\ell$
  colours each. Assign these lists to the vertices in $V_i$, for each~$i$,
  $1\le i\le r$.

  We claim that there is no vertex colouring compatible with these lists,
  and hence $\chi_l(G)>\ell$. Suppose, to the contrary, that there is such
  a colouring. Form an $(r,t)$-preference order, where the $i$th order on
  $[t]$ is determined by how frequently the colours are used on $V_i$. That
  is, in the $i$th order, the member of $[t]$ in relative position $1$ is
  the colour appearing most often on $V_i$ and the member in relative
  position $1/t$ is the colour appearing least often (ties can be broken
  arbitrarily). By Definition~\ref{defn:fg}, there is some colour, green
  say, such that if $x_i$ is the position of green in the $i$th order, then
  $x_i\ge \beta(\alpha)$ for $1\le i\le r$, and $\prod_{i\ne i_x}x_i \ge
  f(r,\beta(\alpha),t)\ge f(r,\beta(\alpha))$. Because the second
  condition implies $x_i\ge f(r,\beta(\alpha))$ for $1\le i\le r$, and
  Theorem~\ref{thm:elem} states $f(r,\beta(\alpha))\ge
  f(r,1/r)=(1/r)^{r-1}$, we have $x_i\ge \zeta$ for $1\le i\le r$.

  Let $X$ be the set of vertices that are coloured green. We can find a
  lower bound for $|X_i|$ as follows. Let $Z$ be the set of colours at or
  below relative position $x_i$ in the $i$th order, that is, $Z$ contains
  green and the colours less popular on $V_i$. Let $X_i^*$ be the set of
  vertices in $V_i$ that are coloured with some colour in~$Z$. By
  definition of the $i$th order, $|X_i|\ge|X_i^*|/|Z|\ge |X_i^*|/t$.

  Now $|Z|=x_it$ because green has relative position~$x_i$, and we know
  $x_i\ge \zeta$.  So, by Lemma~\ref{lem:lists}, at least $nx_i^\ell/4$
  lists lie within $Z$, meaning at least $nx_i^\ell/4$ vertices in $V_i$
  have lists within~$Z$. All of these vertices necessarily choose a colour
  in~$Z$, and so lie within $X_i^*$. Therefore $|X_i^*|\ge nx_i^\ell /4$
  and hence $|X_i|\ge |X_i^*|/t\ge nx_i^\ell /4t$.

  Consequently, using equation~(\ref{eqn:fng}), and writing $f$ for
  $f(r,\beta(\alpha))$, noting that $f\le f(2,0)=1/2$ (see
  Theorem~\ref{thm:elem}), we have
  \begin{align*}
    \prod_{i\ne i_X}|X_i|\ge\prod_{i\ne i_x}{nx_i^\ell \over 4t} \ge
    f^\ell \left({n\over 4t}\right)^{r-1} 
    &=   f^{g(r,\alpha)\log_r d}f^{-6r\log\log   d}
    \left({n\over 4t}\right)^{r-1} \\
    &\ge {n^{r-1}\over d}\left(2^{6\log\log d}\over 4t\right)^{r-1}\,.
  \end{align*}
  Now $2^{6\log\log d}/4t\ge \zeta 2^{6\log\log d}/9\ell^2 \ge
  (1/r)^{r-1}2^{6\log\log d}/10\log^2_r d\ge \log^2 d$. Thus $\prod_{i\ne
    i_X}|X_i|\ge  n^{r-1}(\log^2 d)/d$. But $G$ has property $I(r,n,d)$ and
  so $X$ cannot be an independent set, in contradiction to it being the set
  of green vertices in a proper colouring.
\end{proof}

We remark that, if the set $X$ in this proof were a random set of vertices,
then the proof would work for every $r$-partite $r$-graph~$G$ even without
assuming $I(r,n,d)$, because a random set with the specified lower bounds
on $|X_i|$ would not be independent. In fact the set of vertices whose lists
lie within $Z$ is random, but there seems no reason why the set $X$ itself
should be random.

\section{Random $r$-partite hypergraphs}\label{sec:randg}

We begin with a lemma that we shall use several times when treating various
kinds of random $r$-partite hypergraphs on the vertex set
$V_1\cup\cdots\cup V_r$.

\begin{lem}\label{lem:y}
  Let some probability distribution be given on the space of subsets of
  $V=V_1\cup\cdots\cup V_r$, where $|V|=rn$. Let
  $E$ be some event. Suppose, for each non-empty $X\subset
  V$, that $\Pr(X\in
  E)\le(|X_{i_X}|/2en)^{(r+1)|X_{i_X}|}$ holds.  Then
  $E=\emptyset$ almost surely, as $n\to\infty$.
\end{lem}
\begin{proof}
  There are at most $(q+1)^r\le 2^{rq}$ possibilities for the
  tuple $(|X_1|,\ldots,|X_r|)$ if $|X_{i_X}|=q$, and, for each such
  possibility, the number of possible sets $X$ is
  $$
  \prod_i{n\choose |X_i|}\le \prod_i\left({en\over |X_i|}\right)^{|X_i|}\le
  \left({en\over q}\right)^{rq}\,,
  $$
  because $(e/x)^x$ is an increasing function of~$x$ for $x\le 1$. Hence
  the total probability that there is some set $X\in E$ is at most $
  \sum_{q\ge 1} 2^{rq}(en/q)^{rq}(q/2en)^{(r+1)q}
  =\sum_{q\ge1}(q/2en)^q$. Since $(x/2e)^x$ decreases for~$0<x\le 1$, the
  first $\sqrt n$ terms of this sum add to at most $\sqrt n(1/2en)$, which
  tends to zero. Since $q\le n$, the remaining terms add to at most
  $\sum_{q\ge\sqrt n}(1/2e)^q$, which also tends to zero. Therefore
  $E$ is almost surely empty.
\end{proof}

The proof of Theorem~\ref{thm:randg} involves a routine
verification. In fact, we do slightly more work than we need to, though the
extra effort involved is negligible. We show that $G$ almost surely has the
two stronger properties $I'(r,n,d)$ and $D'(r,n,d)$. Property $I'(r,n,d)$
asserts that every set $X$ containing at most $n/2d^{1/(r-1)}$ edges
satisfies~(\ref{eqn:I}), and Property $D'(r,n,d)$ asserts that every set
$X$ satisfying~(\ref{eqn:D}) is $(4(\log d/\log\log d)-1)$-degenerate.  The
reason for adding this complication is that we can copy over the proof
directly for use again in~\S\ref{sec:regg}.

\begin{proof}[Proof of Theorem~\ref{thm:randg}.]
  Let $G\in \mathcal{G}(n,r,p)$ be a random $r$-partite $r$-uniform
  hypergraph and let $d=pn^{r-1}\ge d_0$.

  Let $X\subset V(G)$ and let $x_i=|X_i|/n$. Let $S$ be the number of edges
  in $G[X]$. Then $S\in{\rm Bi}(\prod_{i=1}^r|X_i|, p)$, having mean
  $\lambda=p\prod^r_{i=1}|X_i|=d|X_{i_X}|\prod_{i\ne i_X}x_i$.

  Let $E$ be the collection of sets $X\subset V(G)$ such that $\prod_{i\ne
    i_X}|X_i|\ge n^{r-1}(\log^2d)/d$ but $G[X]$ has at most
  $n/2d^{1/(r-1)}$ edges. To show that $G$ almost surely has property
  $I'(r,n,d)$, we must show $E=\emptyset$ almost surely, and to do this we
  apply Lemma~\ref{lem:y}. Let $X\subset V$. If $\prod_{i\ne i_X}|X_i|<
  n^{r-1}(\log^2d)/d$ then $X\notin E$ so $\Pr(X\in E)=0$.  If $\prod_{i\ne
    i_X}|X_i|\ge n^{r-1}(\log^2d)/d$ then $\Pr(X\in E)$ is the probability
  that $S\le n/2d^{1/(r-1)}$. In this case, $\prod_{i\ne i_X}x_i\ge
  (\log^2d)/d$ so $\lambda > |X_{i_X}|\log^2d$. Moreover $|X_{i_X}|\ge
  (\prod_{i\ne i_X}|X_i|)^{1/(r-1)}> n/d^{1/(r-1)}$. Hence certainly
  $\Pr(X\in E)\le \Pr(S\le \lambda/2)\le e^{-\lambda/8}$ by
  Proposition~\ref{prop:chernoff}, so $\Pr(X\in E)\le
  e^{-|X_{i_X}|(\log^2d)/8}$.  Therefore, for Lemma~\ref{lem:y} to apply,
  it is enough to show that $e^{-(\log^2d)/8}\le
  (|X_{i_X}|/2en)^{r+1}$. But $|X_{i_X}|\ge n/d^{1/(r-1)}$ so we need only
  show that $e^{-(\log^2d)/8}\le (1/2ed^{1/(r-1)})^{r+1}$, which easily
  holds if $d$ is large.

  Now let $\mathcal{X}=\{X\subset V:\prod_{i\ne i_X}|X_i|\le
  n^{r-1}/d\}$. To show that $G$ almost surely has property $D'(r,n,d)$ we
  must show, almost surely, that every $X\in\mathcal{X}$ is
  $(k-1)$-degenerate, where $k=4\log d/\log\log d$. Notice that if
  $X\in\mathcal{X}$ and $Y\subset X$ then $Y\in\mathcal{X}$, and therefore
  to show every $X\in\mathcal{X}$ is $(k-1)$-degenerate it suffices to show
  that every $X\in\mathcal{X}$ is either empty or has a vertex of degree at
  most~$k-1$. We shall in fact show that if $X\in\mathcal{X}$ and
  $X\ne\emptyset$ then $G[X]$ contains fewer than $k|X_{i_X}|$ edges, and
  so the largest class of $X$ has a vertex of degree less than~$k$.

  So let $E=\{X\in\mathcal{X}:\,X\ne\emptyset,\,S\ge k|X_{i_X}|\}$, where
  $S$ is the number of edges in~$G[X]$. We wish to show that $E=\emptyset$
  almost surely, and we again use Lemma~\ref{lem:y}. Since $S\in{\rm
    Bi}(\prod_{i=1}^r|X_i|, p)$, the probability that $S\ge k|X_{i_X}|$ is
  at most
  $$
  {\prod_i|X_i|\choose k|X_{i_X}|}p^{k|X_{i_X}|}\le\left(ep\prod_i|X_i|\over
    k|X_{i_X}|\right)^{k|X_{i_X}|}=\left(ed\prod_{i\ne i_x}x_i\over
    k\right)^{k|X_{i_X}|}\,,
  $$
  where $x_i=|X_i|/n$. To apply the lemma successfully, we need
  $(ed\prod_{i\ne i_x}x_i/ k)^k\le (|X_{i_X}|/2en)^{r+1}$, or $s\le 1$
  where $s=(2en/|X_{i_X}|)^{r+1}(ed\prod_{i\ne i_x}x_i/ k)^k$. Let
  $z=d^{-1/(r-1)}$. For $|X_{i_X}|\le zn$, we use the inequality $\prod_{i\ne
    i_x}x_i\le (|X_{i_X}|/n)^{r-1}$, and so $s\le
  (2en/|X_{i_X}|)^{r+1}(ed(|X_{i_X}|/n)^{r-1}/k)^k$: this is an increasing
  function of~$|X_{i_X}|$ (we can assume $k>3$ because $d_0$ is large) and
  so $s\le (2e/z)^{r+1}(edz^{r-1}/k)^k = (2e/z)^{r+1}(e/k)^k$. For
  $|X_{i_X}|\ge zn$, we use instead that $\prod_{i\ne i_x}x_i\le 1/d$
  because $X\in\mathcal{X}$, and therefore $s\le
  (2en/|X_{i_X}|)^{r+1}(e/k)^k\le(2e/z)^{r+1}(e/k)^k$.  Consequently
  $s\le(2e/z)^{r+1}(e/k)^k\le (2e)^{r+1}d^3(e/k)^k$ holds for
  every~$X\in\mathcal{X}$, and this bound is less than one because $k=4\log
  d/\log\log d$ and $d_0$ is large. This shows that, almost surely, no
  $X\in\mathcal{X}$ has more than $k|X_{i_X}|$ edges, and almost surely $G$
  has property $D(r,n,d)$.
\end{proof}

\section{Regular $r$-partite hypergraphs}\label{sec:regg}

In this section we aim to prove Theorem~\ref{thm:regg}. Rather than apply
the configuration model, which would work only for $n$ much larger
than~$d^4$, we work instead with the space $\mathcal{H}(n,r,d)$ of
$d$-regular $r$-partite hypergraphs that are the union of $d$ independently
chosen perfect matchings $M_1,\ldots,M_d$. So $M_i$ is a set of $n$ pairwise
disjoint edges, and $M_1,\ldots,M_d$ are chosen uniformly and independently
from all possible matchings. Hypergraphs in
$\mathcal{H}(n,r,d)$ may have multiple edges.

An $r$-graph $H\in\mathcal{H}(n,r,d)$ is unlikely to be simple, but a small
modification of it, $\widehat{H}$, will be simple. Theorem~\ref{thm:regg}
holds if $\widehat{H}$ has properties $I(r,n,d)$ and $D(r,n,d)$; for this
to happen, we require $H$ to satisfy $I'(r,n,d)$ and $D'(r,n,d)$, described
in~\S\ref{sec:randg}.

\begin{lem}\label{lem:IDdash}
  With probability tending to one as $d\to\infty$, $H\in\mathcal{H}(n,r,d)$
  has properties $I'(r,n,d)$ and $D'(r,n,d)$.
\end{lem}
\begin{proof}
  Let $H\in\mathcal{H}(n,r,d)$ be a random $d$-regular $r$-partite
  $r$-uniform hypergraph. Let $X\subset V(H)$ and let $x_i=|X_i|/n$. Let
  $R$ be the number of edges in $H[X]$. Recall that in the proof of
  Theorem~\ref{thm:randg} we studied the distribution of a variable very
  similar to~$R$, namely~$S$, the number of edges in $G[X]$ where $G\in
  \mathcal{G}(n,r,p)$ and $d=pn^{r-1}$. Thus
  $\mathbb{E}S=p\prod_{i=1}^r|X_i|=d|X_{i_X}|\prod_{i\ne i_X}x_i$. When proving
  that $G$ had property $I'(r,n,d)$ we used only that
  $\mathbb{E}S=d|X_{i_X}|\prod_{i\ne i_X}x_i$ and that the bound in
  Proposition~\ref{prop:chernoff} holds for~$S$. We shall show that the
  same bound holds for~$R$, and moreover
  $\mathbb{E}R=\mathbb{E}S$. Therefore the proof that $G$ has $I'(r,n,d)$
  can be used verbatim to show that $H$ has $I'(r,n,d)$.

  Let $Z$ be the random variable that is the number of edges of $M_1$ lying
  inside~$X$. For notational convenience, suppose $X_{i_X}=X_1$. Clearly
  $\mathbb{E}Z=|X_1|\prod_{i=2}^rx_i$, since the edge containing $v\in V_1$
  has probability $\prod_{i=2}^rx_i$ of meeting each $X_i$, $i\ge2$. Now
  $M_1$ can be generated from $r-1$ independent random bijections $V_i\to
  V_1$, $2\le i\le r$, the edge of $M_1$ containing $v\in V_1$ being $v$
  together with those vertices that map to~$v$. So $Z=|X_1\cap
  T_2\cap\cdots\cap T_r|$, where $T_i$ is the image of $X_i$, $2\le i\le
  r$. By Remark~\ref{rem:chernoff}, $Z$ is a sum of independent Bernoulli
  variables. Finally, $R$ is the sum of $d$ independent copies of~$Z$, so
  it too is a sum of independent Bernoulli variables, and hence
  Proposition~\ref{prop:chernoff} holds for~$R$. Moreover
  $\mathbb{E}R=d\mathbb{E}Z=\mathbb{E}S$, and this completes the proof that
  $H$ has $I'(n,r,d)$.


  For the proof that $H$ has $D'(n,r,d)$ we again copy from the proof of
  Theorem~\ref{thm:randg}, and again assume $X_{i_X}=X_1$. Let $T\subset
  X_1$, $|T|=k_1$. The probability that $T\subset T_i$, where $T_i$ is as
  in the previous paragraph, is ${n-k_1\choose
    |X_i|-k_1}{n\choose|X_i|}^{-1}\le x_i^{k_1}$. Thus the probability is
  at most $(\prod_{i=2}^rx_i)^{k_1}$ that, for every $v\in T$, the edge of
  $M_1$ meeting $v$ lies inside $X$. So the probability that $X$ contains
  at least $k_1$ edges of $M_1$ is at most ${|X_1|\choose
  k_1}(\prod_{i=2}^rx_i)^{k_1}$. If $R\ge k|X_1|$, that is, $H[X]$ has at
least $k|X_1|$ edges, then there are numbers $k_1,\ldots,k_d$ with
$k_1+\cdots+k_d=k|X_1|$ such that $X$ has $k_j$ edges of $M_j$, $1\le j\le
d$. Thus $\Pr(R\ge k|X_1|)\le \sum_{k_1+\cdots+k_d=k|X_1|}\prod_{j=1}^d
{|X_1|\choose k_j}(\prod_{i=2}^rx_i)^{k_j} =
{d|X_1|\choose k|X_1|}(\prod_{i=2}^rx_i)^{k|X_1|}\le
((ed/k)\prod_{i=2}^rx_i)^{k|X_1|}$. But this is exactly the same as the
bound on $\Pr(S\ge k|X_{i_X}|)$ that was used in the proof of
Theorem~\ref{thm:randg}, so, copying the rest of the proof verbatim, we
have that $H$ has $D'(n,r,d)$ almost surely.
\end{proof}

The next lemma describes the modification of $H\in\mathcal{H}(n,r,d)$
that produces~$\widehat H$. Because $H$ is close to simple, we can remove
just a few edges to achieve simplicity, and replace them with well-chosen
new edges to preserve regularity.

\begin{lem}\label{lem:hhat}
  There is a number $d_3=d_3(r)$ such that the following holds. Let $d$ be
  an integer with $d\ge d_3$ and let $n\ge r^5d^4$. Then, with probability
  at least~$1/8$, $H\in\mathcal{H}(n,r,d)$ has the following
  property. There is a set $I$ of at most $r^3d^2$ independent (that is,
  pairwise disjoint) edges in~$H$, and a set $I'$ of $|I|$ independent
  edges none of which is in~$H$, such that $H-I+I'$ is $d$-regular and
  simple.
\end{lem}
\begin{proof}
  A pair of edges $\{e,f\}$ with $|e\cap f|\ge 2$ is called a {\em
    butterfly}. The {\em body} of the butterfly is $e\cap f$. The edges $e$
  and $f$ are the {\em wings} of the butterfly. An $r$-graph is simple if
  it has no butterflies. We make a series of assertions, each of which
  holds with probability (conditional on previous assertions) at least
  $7/8$, if $d_3$ is large enough.
  \begin{itemize}
  \item[(i)] Every butterfly $\{e,f\}$ satisfies $|e\cap f|= 2$. This is
    because the expected number of butterflies with $|e\cap f|\ge 3$ is at
    most ${r\choose 3} n^3{d\choose 2}(1/n^2)^2<1/8$. (To see this, let
    $\{u,v,w\}\subset e\cap f$. There are ${r\choose 3}$ ways to choose
    classes $V_i$ for $u,v,w$, $n^3$ ways to choose $\{u,v,w\}$ in the
    classes, and ${d\choose 2}$ ways to choose matchings $M_i$ containing
    $e$ and $M_j$ containing~$f$. The probability that the edge of $M_i$
    containing $u$ also contains $\{v,w\}$ is $1/n^2$, and likewise for
    $M_j$. Similar considerations explain subsequent assertions.)
  \item[(ii)] No two butterflies have the same body. This is because,
    assuming~(i), the expected number of pairs of butterflies $\{e,f\}$ and
    $\{e,g\}$ with $e\cap f=e\cap g$ is at most ${r\choose2}n^2{d\choose 3}
    (1/n)^3<1/8$.
  \item[(iii)] Distinct butterflies have disjoint bodies. For suppose
    butterflies $\{e,f\}$ and $\{g,h\}$ have bodies $\{u,v\}$ and
    $\{u,w\}$, where $v\ne w$ by~(ii). The expected number of such with
    $e=g$ is at most $r{r\choose2}n^3d{d\choose2}(1/n^2)(1/n)^2<1/16$, and
    with $e\ne g$ is at most $r{r\choose2}n^33{d\choose4}(1/n)^4<1/16$.
  \item[(iv)] No two butterflies share a wing. This is because, assuming
    (ii) and (iii), the expected number of pairs of butterflies $\{e,f\}$
    and $\{e,g\}$ is at most
    $3{r\choose4}n^4d{d\choose2}(1/n^3)(1/n)(1/(n-1))<1/8$. Here the
    factors $1/n$ and $1/(n-1)$ arise from $f$ and $g$ containing their
    bodies, allowing for the fact that, conceivably, $f$ and $g$ come from
    the same matching~$M_i$.
  \item[(v)] Distinct butterflies are disjoint. For suppose butterflies
    $\{e,f\}$ and $\{g,h\}$ have $e\cap g\ne\emptyset$. By~(iv) we cannot
    have $|e\cap g|\ge2$, for if $e=g$ then $\{e,f\}$ and $\{g,h\}$ share a
    wing, and if $e\ne g$ then $\{e,g\}$ is also a butterfly sharing a wing
    with both $\{e,f\}$ and $\{g,h\}$, which are distinct. Hence $|e\cap
    g|=\{u\}$ for some vertex~$u$. By~(iii) the expected number of these is
    at most
    $r{r\choose2}^2n^5{d\choose2}^2(1/n^2)^2(1/(n-1))^2<1/8$. Here we
    chose $u$ and the two bodies, followed by $e$ and $g$ and by $f$ and~$h$.
  \end{itemize}

  Let $b$ be the number of butterflies in~$H$. The expected value of $b$ is
  at most ${r\choose2} n^2{d\choose2}(1/n)^2 < r^2d^2/4$. So with
  probability at least~$1/8$, (i)--(v) all hold and $b\le
  r^2d^2<r^3d^2/2$. Let $\{e_1,f_1\},\ldots,\{e_b,f_b\}$ be the
  butterflies. Beginning with $I=I'=\emptyset$, we construct $I$ and $I'$
  in $b$ steps. At the $j$th step, we add two disjoint edges $\{e_j,g_j\}$
  of $H$ to $I$, and add to $I'$ two disjoint edges $\{e_j',g_j'\}$,
  neither of which is in~$H$, and satisfying $e_j\cup g_j=e_j'\cup g_j'$;
  this last property will ensure that $H-I+I'$ is $d$-regular. Property~(v)
  and the choice of $g_j$ will ensure the edges of $I$ are independent, and
  hence so are the edges of~$I'$. 

  To find these edges, consider $\{e_j,f_j\}$. Property~(i) holds so let
  $e_j\cap f_j=\{u,v\}$: for convenience we assume $u\in V_1$ and $v\in
  V_2$. Let $Q$ be the set of vertices in one of $e_1,f_1,\ldots,e_b,f_b$
  or in some edge of $I$ or in some edge containing either $u$ or $v$: then
  $|Q|\le 2br + |I|r+2dr\le 4br+2rd\le 3r^3d^2$. There are at most $|Q|d$
  edges meeting~$Q$, and at most $|Q|drd$ edges meeting these edges. But
  $|Q|drd\le 3r^5d^4<nd$, and $H$ has $nd$ edges. Hence there is an edge
  $g_j$ of $H$ so that no edge of $H$ meets both $g_j$ and~$Q$. Let $x$ and
  $y$ be the vertices of $g_j$ in $V_1$ and $V_2$ respectively, and put
  $e_j'=(e_j\setminus\{u\})\cup\{x\}$,
  $g_j'=(g_j\setminus\{x\})\cup\{u\}$. Since $e_j'$ and $g_j'$ meet $Q$, in
  $v$ and $u$ respectively, and both meet~$g_j$, neither $e_j'$ nor $g_j'$
  is in~$H$.

  By choice of $Q$ and by~(v), $e_j$ and $g_j$ are disjoint, $e_j\cup
  g_j=e_j'\cup g_j'$, and this set of $2r$ vertices is disjoint from
  any edge so far in~$I$ (and hence also disjoint from any edge in~$I'$),
  and is also disjoint from any butterfly. Furthermore, adding $e_j'$ to $H$
  does not create a butterfly: for if $\{e_j',f\}$ is such a butterfly then
  $f$ lies in $H$, $f\cap e_j\ne \emptyset$, $f\ne f_j$, so $|f\cap e_j|=1$
  and $x\in f$, contradicting the choice of~$g_j$. Likewise $\{g_j',h\}$
  cannot be a butterfly, where $h$ is in~$H$, because $\{g_j,h\}$ is not a
  butterfly, implying $u\in h$ and $h\cap g_j\ne\emptyset$, another
  contradiction. So the addition of $\{e_j',g_j'\}$ to $H$ will not create
  a butterfly. Thus after $b$ steps we reach sets $I$ and $I'$, with
  $|I|=|I'|=2b\le r^3d^2$, as described in the lemma.
\end{proof}

\begin{proof}[Proof of Theorem~\ref{thm:regg}]
  Take $H$ satisfying Lemmas~\ref{lem:IDdash} and~\ref{lem:hhat}, and let
  $\widehat{H}=H-I+I'$. By the properties of Lemma~\ref{lem:hhat},
  $\widehat{H}$ is $d$-regular and simple. Let $X$ be an independent set in
  $\widehat{H}$. In $H$, $X$ contains at most $|I|\le r^3d^2$
  edges. Recalling that $n\ge r^5d^4$, this means $H[X]$ has at most
  $n/2d^{1/(r-1)}$ edges and, since $H$ satisfies property $I'(r,n,d)$,
  this means $X$ satisfies~(\ref{eqn:I}). Therefore $\widehat{H}$ has
  property $I(r,n,d)$. Now suppose $X$ is a set
  satisfying~(\ref{eqn:D}). Since $H$ has property $D'(r,n,d)$, this means
  $H[X]$ is $(k-1)$-degenerate, where $k=4(\log\log d)/\log d$. But the
  edges of $\widehat{H}[X]$ not in $H[X]$ are independent, so
  $\widehat{H}[X]$ is $k$-degenerate. Therefore $\widehat{H}$ has property
  $D(r,n,d)$ also.
\end{proof}

\section{More on preference orders}\label{sec:pref}

In this section we aim to establish some basic properties of $f(r,\theta)$
and $g(r,\alpha)$. However the notion of an $(r,m)$-preference order~$P$
and the definition of $f_P(\theta)$ are tailored to suit the proof of
Theorem~\ref{thm:mainthm}, and in themselves are somewhat cumbersome to
work with. The value of $f_P(\theta)$ takes no account of any $x\in P$ with
$x_i<\theta$ for some~$i$, and for every $x\in P$ it takes no account of
$x_{i_x}$, making some information in $\{x:x\in P\}$ appear
redundant. Further, it can be difficult to manipulate simultaneously the
$r$ different orders in~$P$. 

These drawbacks are resolved by introducing the notion of a {\em cover},
which is nothing more than a perfect matching. Complete information about
the function $f(r,\theta)$ can (in principle) be found by studying covers,
without the complication and redundancy of preference orders. Moreover, to
obtain a useful lower bound on $f(r,\theta)$ it is more or less necessary
to work with covers.

\subsection{Preference orders and covers}

\begin{defn}\label{defn:cover}
  For $r\ge1$, an $r$-{\em cover} is an $r$-graph $Q$ with
  $V(Q)\subset[0,1]$ whose edges form a perfect matching: that is,
  $|V(Q)|=rn$ for some $n\in\mathbb{N}$ and the edge set $E(Q)$ of $Q$
  comprises $n$ pairwise disjoint edges. We define
  $$
  h(Q) = \max\{\textstyle\prod_{y\in e} y\,:\,e\in E(Q)\}\,.
  $$
  For $\theta\in[0,1/(r+1))$, we define an $(r,\theta,n)$-{\em cover} to be an
  $r$-cover $Q$ with $V(Q)=\{\theta + (1/(r+1)-\theta)j/n\,:\,j\in[rn]\}$. We
  further define
  \begin{align*}
    h(r,\theta,n)&=\min\{h(Q): Q \mbox{ is an $(r,\theta,n)$-cover}\}\\
    \mbox{and}\quad  h(r,\theta)&=\inf\{h(r,\theta,n)\,:\,n\in\mathbb{N}\}\,.
  \end{align*}
Moreover we define $h(r,1/(r+1))=\lim_{\theta\to (1/(r+1))^-}h(r,\theta)=1/(r+1)^r$.
\end{defn}

Observe that $1/(r+1)$ is always in the vertex set of an
$(r,\theta,n)$-cover (when $j=n$); another way to represent the vertex set
is in the form $\{1/(r+1)+jx: j=-n+1,-n+2,\ldots,(r-1)n\}$ where
$x=(1/(r+1)-\theta)/n$. Evidently $\theta^r < h(r,\theta,n)\le
(\theta+r(1/(r+1)-\theta))^r$ for all~$n$, so $\lim_{\theta\to
  (1/(r+1))^-}h(r,\theta)=1/(r+1)^r$, as asserted in the definition.

Notice some differences between a cover and a preference order. The edges
of $Q$ are unordered subsets whereas $\{x:x\in P\}$ consists of ordered
$r$-tuples. The value $h(Q)$ is the maximum, over {\em all} edges, of the
product of {\em all} numbers that are vertices of the edge. We avoid
numbers we are not interested in by specifying the vertex set of the cover:
thus all the vertices of an $(r,\theta,n)$-cover are larger
than~$\theta$. Covers 
are easier to work with than preference orders, but the two are related.

\begin{thm}\label{thm:h=f}
Let $r\in\mathbb{N}$, $r\ge2$ and let $\theta\in[0,1/r]$. Then
$f(r,\theta)=h(r-1,\theta)$.
\end{thm}

As might be expected, the proof of this theorem comes by somehow merging
the $r$~orders of $P$ into one single cover, removing the redundant
elements and performing small perturbations of the hypergraphs. In this
context, we say that $Q$ and $Q'$ are {\em similar} if $Q$ is an
$(r,\theta,n)$-cover and $Q'$ is the unique $(r,\theta',n)$-cover such that
the bijection $\theta + (1/(r+1)-\theta)j/n\mapsto \theta' +
(1/(r+1)-\theta')j/n$ between $V(Q)$ and $V(Q')$ takes edges of $Q$ to
edges of~$Q'$.

\begin{lem}\label{lem:similar}
Let $Q$ be an $(r,\theta,n)$-cover and $Q'$ be an $(r,\theta',n)$-cover. If
$Q$ and $Q'$ are similar then $|h(Q)-h(Q')|\le r2^r|\theta-\theta'|$.
\end{lem}
\begin{proof}
  We may suppose that $\theta<\theta'$ and, putting
  $\delta=\theta'-\theta$, that $r2^r\delta<1$ else the lemma is trivial.
  Let $\xi:V(Q)\to V(Q')$ be the bijection $\xi(\theta
  +(1/(r+1)-\theta)j/n)= \theta' + (1/(r+1)-\theta')j/n$. If $y=\theta
  +(1/(r+1)-\theta)j/n$, then $\xi(y)=y+\delta-j\delta/n$. Since $j\in[rn]$
  we have $y-r\delta< \xi(y)\le y+\delta$. If $e$ is an edge of $Q$ and
  $e'$ is the corresponding edge of $Q'$ then $\prod_{\xi(y)\in
    e'}\xi(y)\le \prod_{y\in e}(y+\delta)\le \prod_{y\in e}y + 2^r\delta$,
  so $h(Q)\le h(Q')+2^r\delta$. Likewise $\prod_{y\in e}y\le
  \prod_{\xi(y)\in e'}(\xi(y)+r\delta)\le \prod_{\xi(y)\in e'}\xi(y) +
  r2^r\delta$, so $h(Q)\le h(Q')+r2^r\delta$.
\end{proof}

To prove Theorem~\ref{thm:h=f} we first bound $h$ in terms of $f$.

\begin{lem}\label{lem:hltf} Let $r\ge2$, $\theta\in[0,1/r)$ and
  $m\in\mathbb{N}$. Then
  $h(r-1,\theta,n)\le f(r,\theta,m)+(r-1)2^{r-1}/m$ holds, where
  $n=m-r\lceil\theta m\rceil + r$.
\end{lem}
\begin{proof}
  Take a preference order $P$ on $[m]$ with
  $f_P(\theta)=f(r,\theta,m)$. Form an $r$-cover $Q_1$ with vertex set
  $\{i/rm\,:\,i\in[rm]\}$ by merging the $r$ orders of~$P$ but reducing the
  values in the $i$th order by $(i-1)/rm$: that is, for each
  $x=(x_1,\ldots,x_r)\in P$, $Q_1$ has the edge
  $e(x)=\{x_1,x_2-1/rm,x_3-2/rm,\ldots,x_r-(r-1)/rm\}$. Observe that $Q_1$
  is indeed an $r$-cover. Let $k=\lceil \theta m\rceil-1$, so
  $k/m<\theta\le (k+1)/m$. Then the condition $x_i\ge \theta$ for $1\le
  i\le r$ is equivalent to $\min\{v: v\in e(x)\}>k/m$.

  We shall transform $Q_1$ but keep the same vertex set. Let
  $A=\{1/rm,\ldots,k/m\}$ be the $rk$ smallest elements of $V(Q_1)$.  For any
  $r$-cover $Q$ with $V(Q)=V(Q_1)$, let $F(Q)=\{e\in E(Q):\,e\cap
  A=\emptyset\}$. So $e(x)\in F(Q_1)$ if and only if $x_i> \theta$ for
  all~$i$.  For $e\in E(Q)$ let $\psi(e)$ be the product of the $(r-1)$
  elements in $e$ except the largest; then $\psi(e(x))\le \prod_{i\ne i_x}
  x_i$ for $e(x)\in E(Q_1)$. So, defining $\Psi(Q)=\max\{\psi(e):e\in
  F(Q)\}$ we have $\Psi(Q_1)\le f_P(\theta)$.

  Let $B=\{1-1/r+1/rm,\ldots,1\}$ be the $m$ vertices greater than
  $1-1/r$. Suppose $e\cap B=\emptyset$ for some edge~$e$. Since $|B|=m=|E(G)|$
  there must be some edge $f$ with $|f\cap B|\ge 2$. Let $u$ be the
  greatest element of $e$ and $v$ be the second greatest in~$f$. Then
  $u\notin B$ and $v\in B$ so $u<v$. Form $Q'$ from $Q$ by replacing $e$
  and $f$ by $e'=(e\setminus\{u\})\cup\{v\}$ and
  $f'=(f\setminus\{v\})\cup\{u\}$. Then $\psi(e')=\psi(e)$ and
  $\psi(f')\le\psi(f)$, since $u<v$.  Note that $e'\in F(Q')$ only if $e\in
  F(Q)$, and  $f'\in F(Q')$ only if $f\in F(Q)$, so
  $\Psi(Q')\le \Psi(Q)$.  This operation increases the number of edges
  meeting $B$, so, by repeating it as necessary, we
  arrive at an $r$-cover $Q_2$ with $\Psi(Q_2)\le f_P(\theta)$,
  and $|e\cap B|=1$ for every edge $e\in E(Q_2)$.

  Let $C=\{1-1/r-(r-2)k/m+1/rm,\ldots,1-1/r\}$ be the $r(r-2)k$ vertices
  immediately below~$B$. We show there is an $r$-cover $Q_3$ with
  $\Psi(Q_3)\le f_P(\theta)$, $|e\cap B|=1$ for every edge $e\in E(Q_3)$,
  $f\cap C=\emptyset$ for $f\in F(Q_3)$, and $f\subset A\cup B\cup C$ for
  every edge $f\notin F(Q_3)$. If either $k=0$ or $r=2$ we can take
  $Q_3=Q_2$, in the first case because $A=C=\emptyset$ so $F(Q_2)=E(Q_2)$,
  and in the second case because $C=\emptyset$, and $|f\cap B|=|f\cap A|=1$
  for $f\notin F(Q_3)$. So we can assume $k>0$ and $r>2$; that is,
  $C\ne\emptyset$. Suppose that $|f\cap C|<r-2$ for some edge
  $f\notin F(Q_2)$. Since $|C|=r(r-2)k>0$ and there are at most $rk$ edges
  not in $F(Q_2)$ (because each contains a vertex of $|A|$), we have
  $e\cap C\ne\emptyset$ for some edge $e\in F(Q_2)$. Now $|f\cap B|=1$;
  pick some $w\in f\cap A$, and then there exists $u\in f$,
  $u\notin B\cup C$ and $u\ne w$. Let $v\in e\cap C$; then $u<v$. Form
  $Q''$ from $Q_2$ by replacing $e$ and $f$ by
  $e''=(e\setminus\{v\})\cup\{u\}$ and
  $f''=(f\setminus\{u\})\cup\{v\}$. Notice $w\in f''\cap A$ so
  $f''\notin F(Q'')$; also $\psi(e'')\le\psi(e)$ and $e\in F(Q_2)$. Thus
  $\Psi(Q'')\le\Psi(Q_2)$, and the edges in $E(Q'')\setminus F(Q'')$
  contain more vertices of $C$ than do those in $E(Q_2)\setminus F(Q_2)$.
  Hence repeating this operation results in an $r$-cover $Q_3$ with
  $\Psi(Q_3)\le f_P(\theta)$, $|e\cap B|=1$ for all $e\in E(Q_3)$ and
  $|f\cap C|=r-2$ for every edge $f\notin F(Q_3)$. Thus $|f\cap A|=1$ for
  all $f\notin F(Q_3)$. But $|C|=r(r-2)k=(r-2)|A|$ so $C$ lies entirely
  within edges not in $F(Q_3)$; in other words, $|e\cap B|=1$ and
  $e\cap C=\emptyset$ for every $e\in F(Q_3)$.

  Let $V(Q_4)=V(Q)-A-B-C=\{k/m+1/rm,\ldots,1-1/r-(r-2)k/m\}$. Let the edges
  of $Q_4$ be the edges of $F(Q_3)$ with the element in $B$ removed. By the
  properties of~$Q_3$, $Q_4$ is an $(r-1)$-cover. Note
  $|E(Q_4)|=|E(Q_3)|-|A|=m-rk=n$; so in fact, $Q_4$ is precisely an
  $(r-1,k/m,n)$-cover, because $V(Q_4)=\{k/m+jx:j\in[(r-1)n]\}$ where
  $x=1/rm=(1/r-k/m)/n$. By definition of $\Psi(Q_3)$ and of $Q_4$ we see that
  $h(Q_4)=\Psi(Q_3)\le f_P(\theta)$. 

  Finally, let $Q_5$ be the $(r-1,\theta,n)$-cover that is similar to
  $Q_4$. Since $k/m<\theta\le k/m+1/m$, Lemma~\ref{lem:similar} shows $h(Q_5)\le
  h(Q_4)+(r-1)2^{r-1}/m\le f_P(\theta)+(r-1)2^{r-1}/m$, and this proves the
  lemma.
\end{proof}

Now we bound $f$ in terms of $h$. The proof seeks to mimic, as far as
possible, the reverse of the previous proof, though the steps
are now much easier.

\begin{lem}\label{lem:flth}
  Let $r\ge2$, $\theta\in[0,1/r)$ and $n\in\mathbb{N}$. Then
  $f(r,\theta,rm)\le h(r-1,\theta,n)+(r-1)2^{r-1}/m$ holds, where
  $m-r\lceil\theta m\rceil+r=n$.
\end{lem}
\begin{proof}
  Take an $(r-1,\theta,n)$-cover $Q$ with $h(Q)=h(r-1,\theta,n)$. Choose
  $m$ with $n=m-r\lceil\theta m\rceil+r$; such a choice is possible because
  the right hand side increases by at most one as $m$ increases by one. Let
  $Q_1$ be the $(r-1,k/m,n)$-cover that is similar to~$Q$, where
  $k=\lceil \theta m\rceil-1$. By Lemma~\ref{lem:similar}, $h(Q_1)\le
  h(Q)+(r-1)2^{r-1}/m$.

  Now form an $r$-cover $Q_2$ with $V(Q_2)=\{1/rm,\ldots,1\}=V(Q_1)\cup
  A\cup B\cup C$, where $A=\{1/rm,\ldots,k/m\}$,
  $B=\{1-1/r+1/rm,\ldots,1\}$ and $C=\{1-1/r-(r-2)k/m+1/rm,\ldots,1-1/r\}$.
  For each edge $e$ of $Q_1$ let $e\cup\{v\}$ be an edge of $Q_2$, for some
  $v\in B$, and then add $m-n=rk$ further edges each comprising one vertex
  in $A$, one in $B$ and $r-2$ in~$C$. Observe that it is possible to form
  an $r$-cover in this way, because $|V(Q_2)|=rm$, $E(Q_1)=n$, $|A|=rk$,
  $|B|=m$ and $|C|=r(r-2)k$.

  Finally, we form an $(r,rm)$-preference order $P$ from $Q_2$. For each
  edge $f=\{v_1,\ldots,v_r\}\in E(Q_2)$, where $v_1<\ldots<v_r$, let each of
  the $r$-tuples $y_f^1,y_f^2,\ldots,y_f^r$ belong to~$P$, where
  $y_f^i=(v_{1+i},v_{2+i},\ldots,v_{r+i})$, subscripts being evaluated
  modulo~$r$. Note that for each $\ell\in[rm]$ and $i\in[r]$ there is a unique
  $x=(x_1,\ldots,x_r)\in P$ with $x_i=\ell/rm$, and $P$ is indeed an
  $(r,rm)$-preference order. Let $x\in P$ satisfy
  $\prod_{i\ne i_x}x_i=f_P(\theta)$. Then $x=y_f^i$ for some $f\in
  E(Q_2)$. Now $x_i\ge\theta$ for $1\le i\le r$, so $u\ge\theta> k/m$ for all
  $u\in f$. Hence $f\cap A=\emptyset$, so $f=e\cup\{v\}$ for some $e\in E(Q_1)$
  and some $v\in B$. Since $f\cap B=\{v\}$ we have $f(r,\theta,rm)\le
  f_P(\theta)=\prod_{i\ne i_x}x_i=\prod_{z\in f, z\ne v}z=\prod_{z\in e}z\le
  h(Q_1)\le h(Q)+(r-1)2^{r-1}/m$, proving
  the lemma.
\end{proof}

When proving Theorem~\ref{thm:h=f}, we need consider only large $m$ and $n$.

\begin{lem}\label{lem:mk}
  For $r\ge2$, $0\le\theta< 1/r$ and $m,n, k\in\mathbb{N}$,
  $f(r,\theta,km)\le f(r,\theta,m)$ and $h(r-1,\theta,kn)\le h(r-1,\theta,n)$
  hold. In particular, $f(r,\theta)=\liminf_{m\to\infty} f(r,\theta,m)$ and
  $h(r-1,\theta)=\liminf_{n\to\infty} h(r-1,\theta,n)$.
\end{lem}
\begin{proof}
  Take an $(r,m)$-preference order $P$ with
  $f_P(\theta)=f(r,\theta,m)$. Produce an $(r,km)$-preference order $P'$ in
  the following natural way: if $j$ is the number at relative position $x$
  in the $i$th order of~$P$, then place $j,j+m,j+2m,\ldots,j+(k-1)m$ at
  relative positions $x,x-1/km,x-2/km,\ldots,x-(k-1)/km$ in the $i$th order
  of~$P'$.  Then if $x'\in[\theta,1]^r$ and $x'\in P'$, there exists
  $x\in[\theta,1]^r$ with $x\in P$ and $\prod_{i\ne
    i_{x'}}x'_i\le\prod_{i\ne i_x}x_i$, and so $f(r,\theta,mk)\le
  f_{P'}(\theta)\le f_P(\theta)=f(r,\theta,m)$.

  In a similar manner, if $Q$ is an $(r-1,\theta,n)$-cover with
  $h(Q)=h(r-1,\theta,n)$, then we form an $(r-1,\theta,kn)$-cover $Q'$ as
  follows. Note that, by definition, $V(Q)\subset V(Q')$. For each
  $e\in E(Q)$ place the edges $e,e-1/rkn,\ldots, e-(k-1)/rkn$ into $E(Q')$,
  where $e-y=\{x-y:x\in e\}$. It is easy to see that $Q'$ is an
  $(r-1,\theta,kn)$-cover and $h(Q')=h(Q)$.
\end{proof}

\begin{proof}[Proof of Theorem~\ref{thm:h=f}]
  Let $\theta\in[0,1/r)$. By Lemma~\ref{lem:mk} there is a sequence
  $(m_j)_{j=1}^\infty$ with $m_j\to\infty$ and $f(r,\theta,m_j)\to
  f(r,\theta)$. Let $n_j=m_j-r\lfloor \theta m_j\rfloor$. By
  Lemma~\ref{lem:hltf}, $h(r-1,\theta)\le h(r-1,\theta,n_j)\le
  f(r,\theta,m_j)+2^{r-1}/m_j$ holds for all~$j$, and taking the limit as
  $j\to\infty$ gives $h(r-1,\theta)\le f(r,\theta)$. A corresponding
  argument, but using Lemma~\ref{lem:flth}, shows that $f(r,\theta)\le
  h(r-1,\theta)$, so $f(r,\theta)= h(r-1,\theta)$ for $\theta<1/r$. When
  $\theta=1/r$, we have $h(r-1,1/r)=\lim_{\theta\to
    (1/r)^-}h(r-1,\theta)=(1/r)^{r-1}$ by definition. Thus, using the
  result for $\theta<1/r$, we have $\lim_{\theta\to
    (1/r)^-}f(r,\theta)=(1/r)^{r-1}$. But we know (see after
  Definition~\ref{defn:fg}) that $f(r,\theta)$ is decreasing and
  $f(r,1/r)\ge (1/r)^{r-1}$. Therefore $f(r,1/r)=(1/r)^{r-1}=h(r-1,1/r)$,
  completing the proof.
\end{proof}

\subsection{Further properties}\label{subsec:further}

We now establish some basic properties of the functions $f(r,\theta)$ and
$f(r,\theta,m)$, namely continuity, rate of convergence and initial
constancy. In the light of Theorem~\ref{thm:h=f} and Lemmas~\ref{lem:hltf}
and~\ref{lem:flth} we could derive these from corresponding properties of
$h(r-1,\theta)$ and $h(r-1,\theta,n)$, and generally we do so since it
is usually easier to argue in terms of covers than preference
orders.

\begin{lem}\label{lem:cont}
  For $r\ge1$ and $\theta,\theta'\in[0,1/(r+1))$,
  $|h(r,\theta)-h(r,\theta')|\le r2^r|\theta-\theta'|$. In
  particular, $h(r,\theta)$ is continuous for $\theta\in[0,1/(r+1)]$.
\end{lem}
\begin{proof}
  Let $\epsilon>0$. Choose $n$ so that $h(r,\theta,n)<h(r,\theta)+\epsilon$
  and let $Q$ be an $(r,\theta,n)$-cover with $h(Q)=h(r,\theta,n)$. Let
  $Q'$ be the similar $(r,\theta',n)$-cover. By Lemma~\ref{lem:similar},
  $h(r,\theta')\le h(Q')\le h(Q)+r2^r|\theta-\theta'| \le
  h(r,\theta)+r2^r|\theta-\theta'|+\epsilon$. So
  $h(r,\theta)-h(r,\theta')\le r2^r|\theta-\theta'|+\epsilon$, and since
  this holds for all $\epsilon>0$ we have $h(r,\theta)-h(r,\theta')\le
  r2^r|\theta-\theta'|$, The same holds with $\theta$ and $\theta'$
  interchanged, establishing the first half of the lemma, and hence also
  the continuity of $h(r,\theta)$ for $\theta\in[0,1/(r+1))$. But
  $h(r,\theta)$ is continuous at $\theta=1/(r+1)$ by definition of
  $h(r,1/(r+1))=\lim_{\theta\to (1/(r+1))^-}h(r,\theta)$.
\end{proof}

The next lemma bounds how fast $f(r,\theta,m)$ converges to $f(r,\theta)$.
Though we could derive this from a corresponding result for
$h(r-1,\theta,n)$, we need only the bound on $f(r,\theta,m)$, and it is
slightly quicker to prove this directly. The idea of the proof is
straightforward: we choose a large preference order $P'$ with
$f_{P'}(\theta)$ close to $f(r,\theta)$, and from some randomly chosen
elements $y\in P'$ we build an $(r,m)$-preference order $P$ with
$f_P(\theta)$ close to $f(r,\theta)$.

\begin{lem}\label{lem:fdecr}
  For $r\ge2$, $0\le\theta\le 1/r$ and $m\in\mathbb{N}$, $f(r,\theta)\le
  f(r,\theta,m)\le f(r,\theta) + 2^r\sqrt{(\log r m)/m}$ holds.
  In particular $f(r,\theta)=\lim_{m\to\infty}f(r,\theta,m)$.
\end{lem}
\begin{proof}
  The lower bound holds by Definition~\ref{defn:fg}. For the upper bound,
  let $\epsilon > 0$ and choose $N$ with $f(r,\theta,N)\le
  f(r,\theta)+\epsilon$. By Lemma~\ref{lem:mk} we may assume that $N$ is as
  large as we wish, certainly larger than~$m$. Let $P'=(<'_1,\ldots,<'_r)$
  be an $(r,N)$-preference order with $f_{P'}(\theta)=f(r,\theta,N)$. Let
  $S=\{y\in P': y\in[\theta,1]^r\}$. By definition,
  $f_{P'}(\theta)=\max\{\prod_{i\ne i_y}y_i: y\in S\}$.  Since, for
  each~$i$, fewer than $\theta N$ elements $y\in P'$ satisfy $y_i<\theta$,
  we have $|S|> (1-r\theta)N$.

  We now construct an $(r,m)$-preference order~$P=(<_1,\ldots,<_r)$. More
  precisely, we specify only $\{x:x\in P\}$, but this is enough to
  determine $f_P(\theta)$. Put $k=\lceil \theta m\rceil -1$, so
  $k/m<\theta\le (k+1)/m$. Let $q=m-rk$, so $q>m(1-r\theta)$. Partition the
  relative positions into three sets $A=\{1/m,\ldots,k/m\}$,
  $Q=\{(k+1)/m,\ldots,1-(r-1)k/m\}$ and $B=\{1-(r-1)k/m+1/m,\ldots,1\}$, so
  $|A|=k$, $|B|=(r-1)k$ and $|Q|=q$. By definition, $f_P(\theta)=\max\{
  \prod_{i\ne i_x}x_i:x\in P, x\in (Q\cup B)^r\}$.

  Begin by placing $rk$ $r$-tuples $x$ into $P$, so that each $x\in (A\cup
  B)^r$, and for each $x$ there is a unique index~$j$ with $x_j\in A$ and
  $x_i\in B$ for $i\ne j$. It is possible to find such $r$-tuples because
  $|B|=(r-1)|A|$. We finish the construction of $P$ by adding to $P$ a
  further set $R$ of $q$ $r$-tuples (to be described), so that if $x\in R$
  then $x\in Q^r$. Observe that, when this is done,
  $f_P(\theta)=\max\{\prod_{i\ne i_x}x_i: x\in R\}$ holds.

  Note at this point that we may assume that $2^r\sqrt{(\log r m)/m}<1$ and
  in particular $m\ge 4^r$, since otherwise the lemma is trivial because
  $f(r,\theta,m)\le 1$. A further simple observation is that, whatever the
  choice of~$R$,
  $f(r,\theta,m)\le f_P(\theta) =\max\{\prod_{i\ne i_x}x_i: x\in R\}\le
  (1/r+q/m)^{r-1}\le (1/r)^{r-1}+2^{r-1}q/m\le f(r,\theta)+2^{r-1}q/m$ by
  Theorem~\ref{thm:elem}. If, say, $q\le2r$, then using $m\ge 4^r$ we have
  $q/m\le 2r/m < 2/\sqrt m < 2\sqrt{(\log r m)/m}$, and the lemma holds.
  So we may assume that 
  $2r\le q=m-rk\le m -r\theta m + r$, and hence  $1-r\theta\ge r/m$. Since
  $N$ is 
  large we may therefore assume that $|S|>(1-r\theta)N\ge rN/m\ge m\ge q$.

  To find $R$, we turn to the large preference order $P'$, and choose a
  random subset $R'\subset S$ of size~$q$ (we know $|S|>q$). We then take
  $R$ to be the $q$ elements of $Q^r$ whose relative orders are the same as
  those of~$R'$.  Formally, define an injection $\iota:R'\to Q^r$ so that
  if $y\in R'$ and $x=\iota(y)$ then $x_i=(k+j)/m$, where 
  $j=|\{y'\in R': y_i'\le y_i\}|$. Then take $R=\iota(R')$. This completes
  the construction of~$P$.  What remains is to show there is a choice of
  $R'$ such that $f_P(\theta)$ is suitably bounded.

  We say $y\in S$ {\em spoils} $<_i$ if $y\in R'$ and $x_i>y_i +
  (r+\sqrt{2q\log rq})/m$, where $x=\iota(y)\in R$. What is the
  probability that $y$ spoils~$<_i$?  Conditioned on the event $y\in R'$,
  the remaining $q-1$ elements of $R'$ are chosen randomly from
  $S-\{y\}$. Let $X$ be the subset of these taken from the subset $Y\subset
  S$ of elements whose $i$'th co-ordinate exceeds $y_i$: that is,
  $Y=\{y'\in S: y_i'> y_i\}$ and $X=R'\cap Y$. Then $x_i=(k+q-|X|)/m$. Now
  $|X|$ is distributed hypergeometrically with parameters $|S|-1, |Y|,
  q-1$, with mean $\lambda=(q-1)|Y|/(|S|-1)$. Note that, by definition
  of~$S$, there are at most $y_iN-\theta N$ elements of~$S$ not in~$Y$, so
  $|Y|\ge|S|-y_iN+\theta N$.

  Since $x_i=(k+q-|X|)/m$, we have $x_i=(m-(r-1)k-|X|)/m\le
  1-(r-1)\theta+(r-1)/m-|X|/m$. Now $\lambda>(q-1)|Y|/|S|>q|Y|/|S|-1$,
  because $|Y|<|S|$; thus $|Y|<|S|(\lambda+1)/q$. So the inequality
  $|Y|\ge|S|-y_iN+\theta N$ means $y_i\ge |S|/N-|Y|/N+\theta >
  (|S|/N)(1-(\lambda+1)/q) + \theta$. Using $|S|>(1-r\theta)N$ and
  $q>(1-r\theta) m$ this gives $y_i>1-(r-1)\theta-(\lambda+1)/m$.
  Therefore $x_i-y_i<(r-|X|+\lambda)/m$.

  If $y$ spoils $<_i$ then $x_i-y_i>(r+\sqrt{2q\log rq})/m$, and so
  $|X|<\lambda-\sqrt{2q\log rq}$. By Proposition~\ref{prop:chernoff}, the
  probability of this is at most $e^{-(2q\log rq)/2\lambda}<e^{-\log
    rq}=1/rq$.  We say $y$ {\em spoils} $P$ if $y$ spoils $<_i$ for
  some~$i\in[r]$. Thus, conditional on $y\in R'$, the probability that $y$
  spoils $P$ is less than $1/q$. The unconditional probability that $y\in
  R'$ is $q/|S|$, and so the expected number of elements $y\in S$ spoiling
  $P$ is less than $|S|(q/|S|)(1/q)=1$.

  Hence there is some choice of $R'$ for which no element spoils~$P$, and
  $x_i-y_i\le(r+\sqrt{2q\log rq})/m$ for every $y\in R'$ and
  every~$i\in[r]$. We have
  $(r+\sqrt{2q\log rq})/m<r/m + \sqrt{(2/m)\log rm} < 2\sqrt{(\log rm)/m}$
  because $m\ge 4^r$. There is some $x=\iota(y)\in R$ with
  $f_P(\theta)=\prod_{i\ne i_x}x_i$, and so
  $f(r,\theta,m)\le f_P(\theta) =\prod_{i\ne i_x}x_i\le \prod_{i\ne
    i_y}x_i\le \prod_{i\ne i_y}(y_i+ 2\sqrt{(\log rm)/m}) \le \prod_{i\ne
    i_y}y_i+ 2^r\sqrt{(\log rm)/m}$.
  Since $y\in S$, so
  $\prod_{i\ne i_y}y_i\le f_{P'}(\theta)\le f(r,\theta)+\epsilon$, we have
  $f(r,\theta,m)\le f(r,\theta)+\epsilon+ 2^r\sqrt{(\log r m)/m}$. The
  bound holds for every $\epsilon>0$, and so the lemma is proved.
\end{proof}

We now explain why the function $f(r,\theta)$ is constant for small
$\theta$.

\begin{defn}\label{defn:varphi}
For each $r\ge 1$, let $\varphi_r$ be the smallest solution to the equation
$\theta(1-1/(r+1)-(r-1)\theta)^{r-1}=h(r,\theta)=f(r+1,\theta)$.
\end{defn}

Note that there is a solution to this equation, because
$h(r,1/(r+1))=(1/(r+1))^r$, and $h(r,\theta)=f(r+1,\theta)$ by
Theorem~\ref{thm:h=f}. Moreover $h(r,\theta)\ge h(r,1/(r+1))>0$ for all
$\theta\in[0,1/(r+1)]$, so $0<\varphi_r\le 1/(r+1)$.

\begin{thm}\label{thm:varphi}
  For each $r\ge1$, $h(r,\theta)=f(r+1,\theta)$ is constant for
  $\theta\in[0,\varphi_r]$.
\end{thm}
\begin{proof}
  In a nutshell, we take an $h(r,\varphi_r,n)$ cover $Q_0$ with $h(Q_0)\approx
  h(r,\varphi_r)$ and then, given $\theta<\varphi_r$, we increase the vertex set
  $V(Q_0)$ above and below to obtain an $(r,\theta,n+\ell)$-cover
  $V(Q_\ell)$ by adding edges containing the new vertices: the property of
  $\varphi_r$ means that these new edges don't affect $h(Q_\ell)$, so
  $h(Q_\ell)=h(Q_0)$ and hence $h(r,\theta)\le h(r,\varphi_r)$, which is what
  we are after. In practice the outline given needs to be perturbed a
  little, for technical reasons.

  By the continuity of $h(r,\theta)$ (Lemma~\ref{lem:cont}) and the
  definition of~$\varphi_r$, we know that
  $\theta(1-1/(r+1)-(r-1)\theta)^{r-1}<h(r,\theta)$ for $\theta<\varphi_r$.
  Let $\epsilon>0$. Since $h(r,\theta)$ is continuous we may choose
  $0<\theta'<\varphi_r$ with $h(r,\theta')<h(r,\varphi_r)+\epsilon$.  By
  properties of continuity there exists $\delta>0$ such that
  $\theta(1-1/(r+1)-(r-1)\theta)^{r-1}<h(r,\theta)-\delta$ for
  $\theta\in[0,\theta']$. Because $\theta'<1/(r+1)$ there is some
  $(r,\theta', n)$-cover $Q_0$ with $h(Q_0)<h(r,\theta')+\epsilon$ where,
  by Lemma~\ref{lem:mk}, $n$ can be as large as we please. Then
  $V(Q_0)=\{1/(r+1)+jx: j=-n+1,-n+2,\ldots,(r-1)n\}$ with
  $x=(1/(r+1)-\theta')/n$; we choose $n$ so that $x<\delta$.

  Let $\theta\in(0,\theta')$. Choose $\ell$ minimal so that
  $\theta'-\ell(1/(r+1)-\theta')/n\le \theta$, and for $k=0,1,\ldots,\ell$,
  define $\theta_k=\theta'-k(1/(r+1)-\theta')/n$. Thus $\theta_0=\theta'$
  and $\theta_\ell\le \theta$. (Moreover, by increasing $n$ again if
  necessary, we can guarantee that $\theta_\ell>0$.) Observe that
  $(1/(r+1)-\theta_k)/(n+k)=(1/(r+1)-\theta')/n=x$. Hence if $Q_k$ is an
  $r$-cover with $V(Q_k)=\{1/(r+1)+jx: j=-n-k+1,-n+2,\ldots,(r-1)(n+k)\}$,
  then $Q_k$ is an $(r,\theta_k,n+k)$-cover, and $V(Q_0)\subset
  V(Q_1)\subset\cdots\subset V(Q_\ell)$. We construct such covers by
  defining $E(Q_k)=E(Q_{k-1})\cup\{e_k\}$, $k=1,\ldots,\ell$, where
  $e_k=\{1/(r+1)+jx: j=-n-k+1,(r-1)(n+k-1)+1,\ldots,(r-1)(n+k)\}$.

  For each~$k\ge 1$, $\prod_{y\in
    e_k}y<(\theta_k+x)(1-1/r-(r-1)\theta_k)^{r-1} \le
  \theta_k(1-1/(r+1)-(r-1)\theta_k)^{r-1}+x<h(r,\theta_k)$ because
  $x<\delta$, and $h(r,\theta_k)\le h(Q_k)$, because
  $Q_k$ is an $(r,\theta_k,n+k)$-cover. Therefore $h(Q_k)=\max\{\prod_{y\in
    e}y:e\in E(Q_k),e\ne y\}=h(Q_{k-1})$.  Hence $h(r,\theta_\ell)\le
  h(Q_\ell)=h(Q_0)< h(r,\theta')+\epsilon<h(r,\varphi_r)+2\epsilon$. The outer
  inequality holds for all $\epsilon>0$ so $h(r,\theta_\ell)\le
  h(r,\varphi_r)$.  But we know (comment after Definition~\ref{defn:fg}) that
  $f(r,\theta)$ decreases with~$\theta$, meaning by Theorem~\ref{thm:h=f}
  that $h(r,\theta)$ decreases, and so $h(r,\theta_\ell)=
  h(r,\varphi_r)$. Since $\theta_\ell\le \theta<\varphi_r$ and $h$ is decreasing,
  we have $h(r,\theta)= h(r,\varphi_r)$.
\end{proof}

It is readily checked, say by taking logarithms and differentiating, that
the function $\theta(1-1/(r+1)-(r-1)\theta)^{r-1}$ increases for $\theta\le
1/(r^2-1)$ and decreases thereafter. For $r=1$ the function is always
increasing and because $h(1,\theta)$ is decreasing we have
$\varphi_1=1/(r+1)=1/2$. Likewise, for $r=2$, the function is increasing for
$\theta\in [0,1/3]=[0,1/(r+1)]$, and so $\varphi_2=1/(r+1)=1/3$.  Consequently
Theorem~\ref{thm:varphi} means both $h(1,\theta)$ and $h(2,\theta)$ are
constant throughout, as are therefore $f(2,\theta)$ and $f(3,\theta)$
(though we knew this already for other reasons). To get information for
other values of~$r$ we need a useful lower bound on $h(r,\theta)$, which is
what we do next.

\subsection{Lower bounds}

A simple averaging argument provides an initial, but
non-trivial, lower bound on $h(r,\theta)=f(r+1,\theta)$, in terms of the
following function.

\begin{defn}\label{defn:w}
For $r\ge1$ and $\theta\in [0,1/(r+1))$, define
$$
w(r,\theta)=e^{-r}\left(
u(1-1/(r+1)-(r-1)\theta)\over u(\theta)\right)^{1/(1/(r+1)-\theta)}
$$
where $u(y)=y^y$ and $u(0)=1$.
\end{defn}

\begin{lem}\label{lem:gmbound}
  Let $r\ge1$ and $\theta\in [0,1/(r+1))$. Then
  $
  f(r+1,\theta) = h(r,\theta)\ge w(r,\theta)
  $
  holds. In particular, $h(r,0)\ge (r/e(r+1))^r$.
\end{lem}
\begin{proof}
  Let $Q$ be an $(r,\theta,n)$-cover, with $V(Q)=\{\theta+jx:
  j\in[rn]\}$ and $x=(1/(r+1)-\theta)/n$. For $e\in E(Q)$ let
  $\pi(e)=\prod_{y\in e}y$. Then
  $$
  h(Q)=\max_{e\in E(Q)}\pi(e) 
  \ge\Bigl(\prod_{e\in E(Q)}\pi(e)\Bigr)^{1/n} =\Bigl(\prod_{v\in
    V(Q)}v\Bigr)^{1/n} = e^{-S}
  $$
  where $nS=\sum_{v\in V(Q)}\log (1/v)$. Now
  $xnS\le\int_\theta^{1-1/(r+1)-(r-1)\theta}\log(1/t)dt =
  -\log(u(1-1/(r+1)-(r-1)\theta)+ \log
  (u(\theta))+r(1/(r+1)-\theta)$. Hence $h(Q)\ge
  e^{-r}(u(1-1/(r+1)-(r-1)\theta)/u(\theta))^{1/(1(r+1)-\theta)}=w(r,\theta)$
  holds for every $(r,\theta,n)$-cover~$Q$, and, bearing in mind
  Theorem~\ref{thm:h=f} and the definition of $h(r,\theta)$, this proves
  the lemma.
\end{proof}

We explore the properties of $w(r,\theta)$ a little further. The next
definition is close to that of $\varphi_r$ in Definition~\ref{defn:varphi}.

\begin{defn}\label{defn:phi}
  For each $r\ge 1$, let $\phi_r$ be the smallest positive solution to the
  equation $\theta(1-1/(r+1)-(r-1)\theta)^{r-1}=w(r,\theta)$, where
  $w(r,\theta)$ is as in Definition~\ref{defn:w}.
\end{defn}

\begin{lem}\label{lem:w}
Let $r\ge1$. Then the function $w(r,\theta)$ is increasing for
$\theta\le\phi_r$ and decreasing for $\theta\ge\phi_r$.
\end{lem}
\begin{proof}
  It is possible to prove the lemma by just calculating from the
  definitions, but it is more illuminating to interpret the result in terms
  of covers. We argue in a way parallel to the proof of
  Theorem~\ref{thm:varphi}; this time, to avoid excessive technicalities,
  we content ourselves with a detailed sketch.

  Let $n$ be very large and let $Q_0$ be the $(r,0,n)$-cover with
  $V(Q_0)=\{j/(r+1)n: j\in[rn]\}$ and edge set $\{e_k:0\le k< n\}$ where
  $e_0$ comprises the least vertex and $(r-1)$ largest vertices, $e_1$ the
  second least and $(r-1)$ largest remaining vertices, and so on: that is,
  $e_k=\{(k+1)/(r+1)n, (rn-(r-1)k-r+2)/(r+1)n,\ldots,
  (rn-(r-1)k)/(r+1)n\}$. Then $Q_k$, which is $Q_0$ with the edges and
  vertices of $e_0,\ldots,e_{k-1}$ removed, is an $(r, k/(r+1)n,
  n-k)$-cover.

  The product $\prod_{y\in e_k}y$ is very close to
  $p(\theta)=\theta(1-1/(r+1)-(r-1)\theta)^{r-1}$ (the more so as $n$
  grows). Recall from the comment at the end of~\S\ref{subsec:further} that
  $p(\theta)$ increases for $\theta\le 1/(r^2-1)$ and decreases
  thereafter. It must therefore be that $\phi_r< 1/(r^2-1)$ in order for
  Definition~\ref{defn:phi} to be satisfied. In the proof of
  Lemma~\ref{lem:gmbound} we saw that $w(r,\theta)$ was very close to the
  $r$th power of the geometric mean of the vertices of the
  $(r,\theta,n)$-cover~$Q$. Hence if $\theta=k/(r+1)n$ for some $k$ then
  $w(r,\theta)$ is very nearly the $r$th power of the geometric mean of
  $V(Q_k)$. If $\theta<\phi_r$ this quantity is greater than
  $p(\theta)\approx\prod_{y\in e_k}y$ so the mean of $V(Q_{k+1})$ is
  greater than that of $V(Q_k)$; thus $w(r,\theta)$ is increasing at this
  point.

  On the other hand, when $p(\theta)>w(r,\theta)$ then $\prod_{y\in e_k}y$
  exceeds the $r$th power of the geometric mean of $V(Q_k)$, so the mean of
  $V(Q_{k+1})$ will be less than that of $V(Q_k)$ and $w(r,\theta)$ will be
  decreasing. Certainly $p(\theta)>w(r,\theta)$ while
  $\phi_r<\theta<1/(r^2-1)$ since $p(\theta)$ is increasing in this range
  and so $w(r,\theta)$ is perforce decreasing. But now, the fact that
  $p(\theta)$ decreases for $\theta\ge 1/(r^2-1)$ means that $\prod_{y\in
    e_k}y$ is a decreasing function of $k$ in the remaining range; that
  is, when moving from $Q_k$ to $Q_{k+1}$ we are always removing the edge
  with the largest product, so the mean of $V(Q_k)$ continues to decrease,
  and thus so does $w(r,\theta)$.
\end{proof}

In the proof of his lemma it was seen that $p(\theta)$ increases for
$\theta\le\phi_r$. Comparing Definitions~\ref{defn:varphi}
and~\ref{defn:phi}, and noting $h(r,\theta)\ge w(r,\theta)$ as stated in
Lemma~\ref{lem:gmbound}, we then observe that $\phi_r\le\varphi_r$.

\begin{defn}\label{defn:H}
  For $r\ge 1$ and $0\le\theta\le 1/(r+1)$, let
  $$
  H(r,\theta)=
  \begin{cases}
    w(r,\phi_r) & \text{for } \theta\le\phi_r\\
    w(r,\theta) & \text{for } \theta\ge\phi_r
  \end{cases}
  $$
  where $\phi_r$ is as in Definition~\ref{defn:phi}.
\end{defn}

Lemma~\ref{lem:w} means that $H(r,\theta)$ is a decreasing function
of~$\theta$. The importance of $H(r,\theta)$ lies in the next result.

\begin{thm}\label{thm:fge}
Let $r\ge1$ and $0\le\theta\le 1/(r+1)$. Then
$$
f(r+1,\theta)\,=\,h(r,\theta)\,\ge\, H(r,\theta).
$$
\end{thm}
\begin{proof} Theorem~\ref{thm:h=f} shows
  $f(r+1,\theta)=h(r,\theta)$. Lemma~\ref{lem:gmbound} shows
  $f(r+1,\theta)\ge w(r,\theta)$ for all~$\theta$, and it was noted after
  Definition~\ref{defn:fg} that $f(r+1,\theta)$ is decreasing. Thus, for
  $\theta\le\phi_r$, $f(r+1,\theta)\ge f(r+1,\phi_r)\ge w(r,\phi_r) =
  H(r,\theta)$, and for $\theta\ge\phi_r$, $f(r+1,\theta)\ge w(r,\theta) =
  H(r,\theta)$.
\end{proof}

As can be seen from the proofs of Theorem~\ref{thm:varphi} and
Lemma~\ref{lem:gmbound}, what lies behind the bound in the theorem is
this. If $Q$ is an $(r,\theta,n)$-cover, where $n$ is large, and
$\theta<\varphi_r$, then the edge product $\prod_{y\in e}y$ has no effect on
$h(Q)$ if $e$ contains an element less than $\varphi_r$. On the other hand,
if $\theta>\varphi_r$, then $h(Q)$ is near to the lower bound $w(r,\theta)$
only if all edge products are more or less equal.

Surprisingly, it seems that such covers, where all edge products are
roughly equal, might exist. The case of most immediate interest is
$r=3$. In this case, $\phi_3=0.070906\ldots$ and
$w(3,\phi_3)=0.026227\ldots$. Using a computer program to generate
$(3,\phi_3, n)$-covers, which aims to minimise the sum of edge products by
switching pairs of edges in the manner of the proof of
Lemma~\ref{lem:hltf}, we have examples of $(3,\phi_3, 10000)$-covers $Q$
with $h(Q)\le 0.026232\ldots$, meaning $h(3,\phi_3)\le h(3,\phi_3,
10000)\le0.026232\ldots$. Given that $\phi_3\le \varphi_3$ and that
$h(3,\theta)$ is decreasing, this shows $h(3,\varphi_3)\le0.026232\ldots$
and so Theorem~\ref{thm:varphi} implies $h(3,0)\le0.026232\ldots$. But by
Theorem~\ref{thm:fge} we have $h(3,0)\ge
H(3,\phi_3)=w(3,\phi_3)=0.026227\ldots$. In summary, $0.026227\ldots\le
h(3,0)=f(4,0)\le0.026232\ldots$.

Having tried the computer program on a few other pairs $(r,\theta)$,
we are led to make the following conjecture.

\begin{con}\label{con:f=}
  Equality holds in Theorem~\ref{thm:fge} for all $r$ and
  $\theta$.
\end{con}

For what it's worth, we remark that, if true, this conjecture would imply
$\varphi_r=\phi_r$.

\subsection{Proofs of Theorems~\ref{thm:elemg} and~\ref{thm:elem}}

We have already proved most of the properties of $f(r,\theta)$ stated in
Theorem~\ref{thm:elem}; to finish the proof, and to derive
Theorem~\ref{thm:elemg} about $g(r,\alpha)$, we need only add a few more
observations. 

\begin{proof}[Proof of Theorem~\ref{thm:elem}]
  We noted after Definition~\ref{defn:fg} that $f(r,\theta)$ is decreasing,
  and Lemma~\ref{lem:cont} (together with Theorem~\ref{thm:h=f}) shows
  $f(r,\theta)$ is continuous, giving assertion~(a) of the
  theorem. Assertion~(b) was established as part of the proof of
  Theorem~\ref{thm:h=f}. As for~(c), let
  $P=(<_1,\ldots,<_r)$ be an $(r,m)$-preference order and let $P'$ be the
  $(r-1,m)$-preference order $(<_1,\ldots,<_{r-1})$. If
  $x=(x_1,\ldots,x_r)\in P$ and $\min x_i\ge \theta$ then
  $x'=(x_1,\ldots,x_{r-1})\in P'$, and $\prod_{i\ne i_x}x_i\le \prod_{i\ne
    i_{x'}, i\ne r} x_i$, so $f_P(\theta)\le f_{P'}(\theta)$, which
  implies~(c). Assertion~(d) was noted already after
  Definition~\ref{defn:fg}, and again after Theorem~\ref{thm:varphi}.
  
  By Definition~\ref{defn:varphi}, Theorem~\ref{thm:varphi} and
  Lemma~\ref{lem:gmbound}, we have $\varphi_{r-1}(1-1/r)^{r-2}\ge
  \varphi_{r-1}(1-1/r-(r-2)\varphi_{r-1})^{r-2}=h(r-1,\varphi_{r-1})=h(r-1,0)\ge
  ((r-1)/er)^{r-1}$, so $\varphi_{r-1}\ge (1-1/r)e^{-r+1}$. In the light of
  Theorem~\ref{thm:varphi}, assertion~(e) follows. Assertion~(f) was
  explained immediately before Conjecture~\ref{con:f=}. The first
  inequality of assertion~(g) is part of Lemma~\ref{lem:gmbound}, given
  that $f(r,0)=h(r-1,0)$. For the second inequality, consider the
  $(r-1,0,n)$-cover $Q$ with $V(Q)=\{j/rn: j\in[(r-1)n]\}$ and edge set
  $E(Q)=\{e_k:k\in[n]\}$ where $e_k=\{k/rn, (k+n)/rn,
  (k+(r-2)n)/rn\}$. Then $h(Q)=\prod_{y\in e_n}y=(r-1)!/r^{r-1}$, and
  $f(r,0)=h(r-1,0)\le h(Q)$. This completes the proof.
\end{proof}

\begin{proof}[Proof of Theorem~\ref{thm:elemg}]
  We appeal throughout to the properties of $f(r,\theta)$ given in
  Theorem~\ref{thm:elem} and to the fact that
  $g(r,\alpha)=-1/\log_rf(r,\beta(\alpha))$.

  By Definition~\ref{defn:g}, $\beta(\alpha)^\alpha=f(r,\beta(\alpha))$ so,
  by Theorem~\ref{thm:elem}~(a), $\beta(\alpha)$ is continuous in $\alpha$
  and, by the remark following the definition, strictly increasing.  Again
  appealing to Theorem~\ref{thm:elem}~(a) we see that $g(r,\alpha)$ is
  continuous and decreasing, which is assertion~(a). Assertion~(b) is a
  consequence of Theorem~\ref{thm:elem}~(b) and assertion~(c) follows from
  Theorem~\ref{thm:elem}~(d).

  Let $r\ge4$. Define $\delta$ by
  $\varphi_{r-1}^{1+\delta}=f(r,\varphi_{r-1})$. By
  Definition~\ref{defn:g}, $\beta(1+\delta)=\varphi_{r-1}$; by
  Theorem~\ref{thm:varphi} and the fact that $\beta(\alpha)$ is increasing,
  $g(r,\alpha)$ is constant for $\alpha\le 1+\delta$, so to prove
  assertion~(d) it is enough to show that $\delta>1/(r+3)$. In the previous
  proof we showed that $\varphi_{r-1}(1-1/r)^{r-2}\ge
  h(r-1,\varphi_{r-1})=f(r,\varphi_{r-1})$, so $\varphi_{r-1}^\delta\le
  (1-1/r)^{r-2}$. We also showed $\varphi_{r-1}\ge (1-1/r)e^{-r+1}$. Hence
  $[(1-1/r)e^{-r+1}]^\delta\le (1-1/r)^{r-2}$, and so
  $e^{-\delta(r-1)}\le(1-1/r)^{r-2-\delta}< e^{(r-2-\delta)/r}$, or
  $-\delta(r-1)< (r-2-\delta)/r$. Thus $\delta> (r-2)/(r^2-r+1)$, and,
  since $r\ge4$, this implies $\delta>1/(r+3)$ as desired.

  Assertion~(e) follows from the bounds $0.026227\le f(4,0)\le0.026233$
  mentioned before Conjecture~\ref{con:f=}, and~(f) follows straightaway from
  Theorem~\ref{thm:elem}~(g), which completes the proof.
\end{proof}

\section{Property B}\label{sec:B}

An $\ell$-uniform hypergraph $H$ is $k$-colourable if its vertices can be
coloured with $k$ colours so that no edge is monochromatic, and $\chi(H)$
is the smallest $k$ for which $H$ is $k$-colourable. Erd\H{o}s~\cite{E1,E2}
studied the minimum number of edges in a bipartite hypergraph~$H$ --- that
is, $\chi(H)=2$: such hypergraphs are said to have ``Property~B''.

Let $m(\ell,r)$ be the minimum number of edges in an $\ell$-graph $H$ with
$\chi(H)>r$. Let $Q(r,\ell)$ be the minimum number of vertices in an
$r$-partite $r$-graph $G$ with list chromatic number $\chi_l(G)\ge\ell$.
Extending the result of Erd\H{o}s, Rubin and Taylor~\cite{ERT}, who proved
the case $r=2$, Kostochka~\cite{K} proved that $m(\ell,r)$ and $Q(r,\ell)$
are closely tied: indeed $m(\ell,r)\le Q(r,\ell)\le rm(\ell,r)$.

There has been no significant improvement on the upper bound for
$m(\ell,2)$ since Erd\H{o}s~\cite{E2} proved $m(\ell,2)\le
\ell^22^\ell$. The lower bound has been improved a few times, the best to
date being $m(\ell,2)=\Omega((\ell/\log \ell)^{1/2} 2^\ell)$ by Radhakrishnan and
Srinivasan~\cite{RS}. A simple proof of this bound, and of the generalisation
$m(\ell,r)=\Omega((\ell/\log \ell)^{1-1/r} r^\ell)$, was given by
Cherkashin and Kozik~\cite{CK}.

The method of~\cite{CK} is close to that of Pluh\'ar~\cite{P}. If an
$\ell$-graph $H$ has fewer edges than is stated in the bound, then a random
argument shows there is some ordering of the vertices without any chain of
edges $e_1,e_2,\ldots,e_r$, such that the last vertex of $e_i$ is the first
of $e_{i+1}$, $1\le i\le r-1$. A simple greedy colouring algorithm then colours
$H$ with $r$ colours.

The relevant part of the proof in~\cite{ERT} and~\cite{K} that relates
$m(\ell,r)$ to $Q(r,\ell)$ is as follows: let $G$ be a complete $r$-partite
$r$-graph with $|E(H)|$ vertices in each class. Consider $V(H)$ to be a
palette and let $E(H)$ be assigned as lists to each vertex in~$V_i$,
$1\le i\le r$. If $G$ can be coloured from these lists then $\chi(H)\le
r$. Any list colouring algorithm can thus be translated to give some lower
bound on $m(\ell,r)$.

Our colouring algorithm for complete $r$-partite $r$-graphs selects some
preference order $P$, after which each vertex $v\in V_i$ chooses the colour
in $L(v)$ most preferred by~$<_i$. In the case $r=2$, where we choose $<_1$
to be the identity and $<_2$ to be its reverse, the translation is to find
an ordering of the vertices of $H$ without a chain $e_1,e_2$ and then to
colour the first vertex of each edge red and the last blue. This is not
quite the same as the method of~\cite{CK} but is effectively equivalent, and
the bound obtained on $m(\ell,2)$ is the same.

However our method makes no use of the fact that the lists in each $V_i$
are the same, and for $r>2$ the translated method is less effective than
the method in~\cite{CK}, though it does show $m(\ell,r)=\Omega((\ell/\log
\ell)^{1/2} r^\ell)$.

\section{Acknowledgement}

The authors are greatly indebted to the referee for a very careful reading
of such technical material.


\begin{thebibliography}{99}
\bibitem{A2} N.~Alon, Degrees and choice numbers, {\sl Random Structures and
  Algorithms} {\bf 16} (2000), 364--368.

\bibitem{AKr} N.~Alon and M.~Krivelevich, The choice number of random
  bipartite graphs, {\sl Ann.\ Comb.} {\bf 2} (1998), 291--297.

\bibitem{AK1} N.~Alon and A.~Kostochka, Hypergraph list coloring and
  Euclidean Ramsey theory. {\sl Random Structures and
  Algorithms} {\bf 39} (2011), 377--390.

\bibitem{AK2} N.~Alon and A.~Kostochka, Dense uniform hypergraphs have high
  list chromatic number, {\sl Discrete Math.} {\bf 312} (2012), 2119--2125.

\bibitem{BMS} J.~Balogh, R.~Morris and W.~Samotij, Independent sets in
  hypergraphs, {\sl J.\ Amer.\ Math.\ Soc.} {\bf 28} (2015), 669--709.

\bibitem{CK} D.~Cherkashin and J.~Kozik, A note on random greedy coloring
  of uniform hypergraphs,  {\sl Random Structures and
  Algorithms} {\bf 47} (2015), 407--413.

\bibitem{E1} P.~Erd\H{o}s, On a combinatorial problem, {\sl Nordisk Mat.\
    Tidskrift} {\bf 11} (1963), 5--10.

\bibitem{E2} P.~Erd\H{o}s, On a combinatorial problem II, {\sl Acta Math.\
    Hungar} {\bf 15} (1964), 445--447.

\bibitem{ERT} P.~Erd\H{o}s, A.L.~Rubin and H.~Taylor, Choosability in graphs,
  {\sl Proc West Coast Conf.\ on Combinatorics, Graph Theory and Computing},
  Congressus Numerantium XXVI (1979), 125--157.

\bibitem{H} L.H.~Harper, Stirling behavior is asymptotically normal, {\sl
    The Annals of Mathematical Statistics} {\bf 38} (1967), 410--414.

\bibitem{HP} P.~Haxell and M.~Pei, On list coloring Steiner triple systems,
  {\em J.\ Combinatorial Designs} {\bf 17} (2009), 314-322.

\bibitem{HV} P.~Haxell and J. Verstra\"ete, List coloring hypergraphs, {\em
    Electr.\ J.\ Combinatorics} {\bf 17} (2010) R129, 12pp.

\bibitem{Hoef} W.~Hoeffding, Probability inequalities for sums of bounded
  random variables, {\sl J.\ of the American Statistical Association} {\bf
    58} (301), 13--30.

\bibitem{HT} R.~H\"aggkvist and A.~Thomason, Oriented hamilton cycles in
  digraphs, {\sl Journal of Graph Theory} {\bf 19} (1995) 471--479.

\bibitem{JLR} S.~Janson, T.~\L{u}czak and A.~Ruci\'nski, {\sl Random graphs}
 (2000), Wiley.

\bibitem{K} A.~Kostochka, On a theorem of Erd\H{o}s, Rubin, and Taylor on
  choosability of complete bipartite graphs, {\sl Electron. J.\ Combin.}
  {\bf 9} (2002), Note~9.

\bibitem{P} A.~Pluh\'ar, Greedy colorings of uniform hypergraphs, {\sl
    Random Structures and Algorithms} {\bf 35} (2009), 216--221.

\bibitem{RS} J.~Radhakrishnan and A.~Srinivasan, Improved bounds and
  algorithms for hypergraph 2-colouring, {\sl Random Structures and
  Algorithms} {\bf 16} (2000), 4--32.

\bibitem{Sap1} A.A.~Sapozhenko, On the number of connected subsets with
  given cardinality of the boundary in bipartite graphs. (in Russian) {\sl
    Metody Diskret.\ Analiz.} {\bf 45} 45 (1987), 42--70, 96.

\bibitem{Sap2} A.A.~Sapozhenko, On the number of independent sets in
  extenders, {\sl Discrete Math.\ Appl.} {\bf 11} (2001), 155--161.

\bibitem{Sap6} A.A.~Sapozhenko, On the number of sum-free sets in Abelian
  groups, {\sl Vestnik Moskovskogo Universiteta, ser.\ Math., Mech.} {\bf 4}
    (2002), 14--18.

\bibitem{Sap5} A.A.~Sapozhenko, Systems of containers and enumeration
  problems, in SAGA 2005, {\sl Lecture Notes in Computer Science}, Springer
  (2005), 1--13. 

\bibitem{Sap3} A.A.~Sapozhenko, The Cameron-Erd\H{o}s conjecture, {\sl
    Discrete Math.} {\bf 308} (2008), 4361--4369. 

\bibitem{ST1} D.~Saxton and A.~Thomason, List colourings of regular
  hypergraphs, {\sl Combinatorics, Probability and Computing} {\bf 21}
  (2012), 315--322.

\bibitem{ST2} D.~Saxton and A.~Thomason, Hypergraph containers,
{\sl Inventiones Mathematicae} {\bf 201} (2015), 925--992.

\bibitem{ST3} D.~Saxton and A.~Thomason, Simple containers for simple
  hypergraphs, {\sl Combinatorics, Probability and Computing} {\bf 25}
  (2016), 448--459.

\bibitem{ST4} D.~Saxton and A.~Thomason, Online containers for hypergraphs,
  with applications to linear equations, {\sl J.\ Combinatorial Theory
    Ser.~B\/} {\bf 121} (2016) 248--283.

\bibitem{VM} V.A.~Vatutin and V.G.~Mikhailov, Limit theorems for the number
of empty cells in an equiprobable scheme for group allocation of particles,
{\sl Theory of Prob.~and its Applications}, {\bf 27} (1982) 734--743.

\bibitem{V} V.G. Vizing, Coloring the vertices of a graph in prescribed
  colors, {\sl Diskret.\ Analiz} No.~29, Metody Diskret.\ Anal.\ v Teorii
  Kodov i Shem 101 (1976), 3--10, 101 (in Russian).

\end{thebibliography}
\end{document}